\documentclass[9pt,reqno]{amsart}
\allowdisplaybreaks
\usepackage{amsmath,amstext,amssymb,epsfig}
\usepackage{amssymb}
\usepackage{tikz-cd}
\usepackage{graphicx}
\usepackage{mathrsfs}
\usepackage{xcolor}
\usepackage[a4paper, margin=1in]{geometry} 
\calclayout
\makeatletter
\renewcommand{\@seccntformat}[1]{\bf\csname the#1\endcsname.}
\renewcommand{\section}{\@startsection{section}{1}
	\z@{.7\linespacing\@plus\linespacing}{.5\linespacing}
	{\normalfont\upshape\bfseries\centering}}
\renewcommand{\@biblabel}[1]{\@ifnotempty{#1}{#1.}}
\makeatother
\theoremstyle{plain}
\newtheorem{theorem}{Theorem}[section]
\newtheorem{lemma}{Lemma}
\newtheorem{proposition}{Proposition}

\theoremstyle{definition}
\newtheorem{example}{Example}
\newtheorem{definition}{Definition}

\usepackage{tikz-cd}
\usepackage[parfill]{parskip}
\usepackage[german]{varioref}
\usepackage[all]{xy}
\usepackage{cancel}
\usepackage{stmaryrd}

\usepackage{tikz-cd}

\usepackage{tikz}
\usetikzlibrary{shapes.geometric, arrows, positioning}

\tikzstyle{startstop} = [rectangle, rounded corners, minimum width=3cm, minimum height=1cm,text centered, draw=black, fill=red!30]
\tikzstyle{process} = [rectangle, minimum width=3cm, minimum height=1cm, text centered, draw=black, fill=orange!30]
\tikzstyle{decision} = [diamond, minimum width=3cm, minimum height=1cm, text centered, draw=black, fill=green!30]
\tikzstyle{arrow} = [thick,->,>=stealth]
\usepackage{float}

\def\>{\succ}
\def\<{\prec}

\def\l{\lambda}

\begin{document}
	\title[Sania Asif \textsuperscript{1}, Zhixiang Wu\textsuperscript{2*}
    ]{Higher Structures of Rota--Baxter Lie $H$-Pseudoalgebras}
\author{Sania Asif\textsuperscript{1}, Zhixiang Wu\textsuperscript{2*}
}
\address{\textsuperscript{1} Institute of Mathematics, Henan Academy of Sciences, Zhengzhou, 450046, P.R. China.}
\address{\textsuperscript{2*}School of Mathematical Sciences, Zhejiang University, Hangzhou, Zhejiang Province, 310058, P.R. China} 
\email{\textsuperscript{1}11835037@zju.edu.cn}
\email{\textsuperscript{2*}wzx@zju.edu.cn}
 \keywords{ Rota--Baxter Lie $H$-pseudoalgebra; Cohomology; Homotopy; Non-abelian extension; Automorphism; Wells map.}
\subjclass[2020]{Primary 17B56, 17B69, 17B38,  Secondary 18N25, 18N40.}
\date{\today}
%

\begin{abstract} This paper investigates Rota–Baxter Lie $H$-pseudoalgebras. We develop a cohomology theory for $\lambda$-weighted relative Rota--Baxter operators via a Maurer--Cartan approach, constructing the underlying differential graded Lie algebra. We classify non-abelian extensions using second cohomology and derive the Wells exact sequence to address the inducibility of automorphisms. Furthermore, we explore the homotopy theory of these structures by introducing $2$-term skeletal and strict Rota--Baxter $L_\infty$-$H$-pseudoalgebras. In particular, we establish a one-to-one correspondence between strict $2$-term structures and crossed modules of Rota--Baxter Lie $H$-pseudoalgebras. These results establish a foundational framework for future advancements in the higher categorical theory of pseudoalgebras with algebraic operators. Ultimately, this work provides a robust foundation for the higher categorical study of pseudoalgebras equipped with algebraic operators. \end{abstract}

\footnote{The corresponding author's emails: ${}^*$ wzx@zju.edu.cn; 11835037@zju.edu.cn.}
\maketitle
 \section{Introduction}
\label{sec:intro}

The study of algebraic structures provides the foundational language for understanding symmetry, geometry, and dynamics across various fields of mathematics and mathematical physics. In the context of two-dimensional conformal field theory and vertex algebras, Lie conformal algebras emerged as a fundamental object \cite{kac1998}. These algebras encode the singular part of the operator product expansion (OPE) of chiral fields. The axiomatic framework of Lie conformal algebras provides a powerful algebraic setting for studying infinite-dimensional Lie algebras, such as the Virasoro and affine Kac-Moody algebras. Formally, a Lie conformal algebra is a $\mathbb{C}[\partial]$-module $L$ equipped with a $\mathbb{C}$-bilinear $\lambda$-bracket satisfying conformal sesquilinearity, skew-symmetry, and the Jacobi identity.

A significant categorical generalization of this theory was introduced by Bakalov, D'Andrea, and Kac \cite{Bakalov-DAndrea-Kac-Theory-of-finite-pseudoalgebras}, who developed the concept of Lie $H$-pseudoalgebras. This framework generalizes Lie conformal algebras by replacing the polynomial ring $\mathbb{C}[\partial]$ with an arbitrary cocommutative Hopf algebra $H$, placing the theory within the pseudotensor category $\mathcal{M}^*(H)$. When $H = \mathbb{C}[\partial]$, the theory recovers Lie conformal algebras; when $H = U(\mathfrak{d})$ for a finite-dimensional Lie algebra $\mathfrak{d}$, it yields multi-dimensional generalizations relevant to higher-dimensional field theories. The foundational work \cite{Bakalov-DAndrea-Kac-Theory-of-finite-pseudoalgebras} established a comprehensive structure theory for finite simple and semisimple Lie $H$-pseudoalgebras, connecting the pseudo-world to classical Lie algebras via the annihilation algebra functor. For further developments on pseudoalgebras, see \cite{Asif-Wu-2025, Boyallian-Liberati-On-pseudo-bialgebras, Chen-Lu-Wei-2025, Das-Nonabelian-cohomology, Gai-Wang-2025, Wu-Cohomology-$H$-pseudoalgebras}. 

Concurrently, the theory of Rota--Baxter algebras has become a central topic in mathematics. Originating from probability theory and combinatorics \cite{Baxter1960, Rota1969}, a Rota--Baxter operator of weight $\lambda$ on an algebra $(A, \circ)$ is a linear map $T: A \to A$ satisfying the identity $T(a) \circ T(b) = T(T(a) \circ b + a \circ T(b) + \lambda(a \circ b))$. This identity plays a crucial role in the algebraic approach to renormalization in perturbative quantum field theory \cite{Connes2000} and leads to the construction of pre-Lie algebras, which are fundamental in the study of operads and flat connections \cite{Guo2012, Li-Hou-Bai-2007}. The interaction between Rota--Baxter operators and various algebraic structures has led to rich cohomology and deformation theories \cite{Das2022, Guo-Keigher-2008, Liu-Wang-2020, Mondal-Saha-2025, Wang-Zhou-2021}.

In this paper, we unify these frameworks by developing a comprehensive cohomological and homotopy-theoretic structure for Rota--Baxter Lie $H$-pseudoalgebras. While cohomology theories for Rota--Baxter operators on classical Lie algebras have been established in \cite{Das2022}, extending these results to the pseudoalgebra setting requires overcoming significant technical challenges related to the pseudotensor category structure and the $H$-linearity of the maps. 

The main results of this paper can be summarized as follows:
\begin{enumerate}
    \item We construct a differential graded Lie algebra (dgLA) whose Maurer--Cartan elements are precisely the $\lambda$-weighted relative Rota--Baxter operators on Lie $H$-pseudoalgebras. This characterization naturally yields a twisting procedure and the associated cochain complex (Theorem \ref{thm:MC-Rota--Baxter}).
    \item We develop a non-abelian cohomology theory $H^2_{\mathrm{nab}}(L_T, M_S)$ for Rota--Baxter Lie $H$-pseudoalgebras, proving a bijection between this cohomology set and the equivalence classes of non-abelian extensions (Theorem \ref{thm:classification}).
    \item To investigate the inducibility of automorphisms in non-abelian extensions, we construct the Wells map $\mathfrak{W}: \mathrm{Aut}(M_S) \times \mathrm{Aut}(L_T) \to H^2_{\mathrm{nab}}(L_T, M_S)$. This leads to the derivation of the Wells exact sequence, which provides a section-independent criterion for inducibility (Theorems \ref{thm6.4} and \ref{thm:wells-rb}).
    \item Finally, we introduce 2-term skeletal and strict Rota--Baxter $L_\infty$-$H$-pseudoalgebras, and prove that strict 2-term structures are in one-to-one correspondence with crossed modules of Rota--Baxter Lie $H$-pseudoalgebras (Theorem \ref{thm:cross_rb}).
\end{enumerate}
 
The paper is organized as follows. Section \ref{sec:prelim} recalls the necessary background on Lie $H$-pseudoalgebras and Rota--Baxter operators. Section \ref{sec:cohomology} develops the Maurer--Cartan characterization and the cohomology of $\lambda$-weighted relative Rota--Baxter operators. Section \ref{sec:extensions} classifies non-abelian extensions, while Section \ref{sec:automorphisms} derives the Wells exact sequence. Finally, Section \ref{sec:homotopy} explores the homotopy theory and crossed modules.

Throughout the paper, all vector spaces, tensor products, and modules are defined over a field $\mathbb k$ of characteristic zero. The symbol $\otimes$ denotes the tensor product over $\mathbb k$, while $\otimes_H$ denotes the tensor product over the cocommutative Hopf algebra $H$. All modules are assumed to be left $H$-modules unless otherwise stated.

\section{Preliminaries on Rota--Baxter Lie $H$-Pseudoalgebras}\label{sec:prelim}
In this section, we establish the necessary background on Lie $H$-pseudoalgebras, their representations, and their cohomology theory. We then introduce Rota--Baxter operators on these structures, laying the foundation for the main results of the paper. 

\subsection{Lie $H$-Pseudoalgebras and Their Representations}
Let $H$ be a cocommutative Hopf algebra over a field $\mathbb{k}$ of characteristic zero, with comultiplication $\Delta$, counit $\varepsilon$, and antipode $S$. We use Sweedler's notation for the coproduct $\Delta(h) = h_{(1)} \otimes h_{(2)}$. Let $M^*(H)$ denote the pseudotensor category associated with $H$, as defined in \cite{Bakalov-DAndrea-Kac-Theory-of-finite-pseudoalgebras}. The objects of $M^*(H)$ are left $H$-modules, and the space of polylinear maps from a collection of modules $\{L_i\}_{i \in I}$ to a module $M$ is given by 
\begin{align*}
    \mathrm{Lin}(\{L_i\}_{i \in I}, M) := \mathrm{Hom}_{H^{\otimes I}}\left(\boxtimes_{i \in I} L_i, H^{\otimes I} \otimes_H M\right),
\end{align*}
where $\boxtimes_{i \in I} L_i$ denotes the tensor product $\bigotimes_{i \in I} L_i$ equipped with the natural $H^{\otimes I}$-module structure.

\begin{definition} {\cite[Definition 3.13]{Bakalov-DAndrea-Kac-Theory-of-finite-pseudoalgebras}}
A \emph{Lie $H$-pseudoalgebra} is a Lie algebra object in the pseudotensor category $M^*(H)$. Explicitly, it is a left $H$-module $L$ equipped with a $H^{\otimes 2}$-linear map ${[\cdot * \cdot]}_L : L \otimes L \to H^{\otimes 2} \otimes_H L$, called the \emph{pseudobracket}, that satisfies the following conditions for all $x, y, z \in L$:
\begin{enumerate}
    \item  Skew-symmetry: ${[x* y]}_L = -(\sigma \otimes_H \mathrm{id}_L){[y* x]}_L$, where $\sigma: H^{\otimes 2} \to H^{\otimes 2}$ is the transposition map.
    \item  Jacobi identity: ${[x * {[y * z]}_L]}_L = {[{[x* y]}_L * z]}_L + \left(\sigma  \otimes_H \mathrm{id}_L\right){[y * {[x * z]}_L]}_L$.
\end{enumerate}
The $H^{\otimes 2}$-linearity means that ${[fx * gy]}_L = ((f \otimes g) \otimes_H \mathrm{id}_L){[x* y]}_L$, for all $f, g \in H$. This operation can be generalised as follow
$$[Fx*Gy]=(F\Delta^n\otimes G\Delta^m\otimes_H\mathrm{id}_L)[x*y]$$ for $F\in H^{\otimes n+1},$  $G\in H^{\otimes m+1}$ and $x,y\in L$.
\end{definition}
\begin{example}
When $H = \mathbb{k}[\partial]$, a Lie $H$-pseudoalgebra is a \emph{Lie conformal algebra}. For example, the Virasoro conformal algebra $\mathrm{Vir}$ is generated by an element $L$ with the $\lambda$-bracket $[L_\lambda L] = (\partial + 2\lambda)L$, or equivalently, $[L*L]=(\partial\otimes 1-1\otimes\partial)\otimes _HL$.
\end{example}

\begin{definition}{\cite[Definition 3.3]{Boyallian-Liberati-On-pseudo-bialgebras}}
Let $L$ be a Lie $H$-pseudoalgebra. A \emph{representation} of $L$ (or an $L$-module) is a left $H$-module $M$ equipped with a $H^{\otimes 2}$-linear map $\rho: L \otimes M \to H^{\otimes 2} \otimes_H M$, defined by $x * u$, such that for all $x, y \in L$ and $u \in M$,
\begin{align*}
  {[x * y]}_L * u  = x * (y * u) - ((12)\otimes_H \mathrm{id}_M)(y * (x * u)) .
\end{align*}
\end{definition}
For any two $H$-modules $M$ and $N$, the space $\mathrm{Chom}(M, N) := \mathrm{Hom}_H(M, H \otimes N)$ plays a crucial role. When $N = \mathbb{k}$, $\mathrm{Chom}(M, \mathbb{k})$ is called the \emph{dual} of $M$.

    Let $(L, {[\cdot * \cdot]}_L)$ be a Lie $H$-pseudoalgebra. The \emph{adjoint representation} of $L$ is the representation of $L$ on itself given by the pseudobracket ${[\cdot * \cdot]}_L$. We denote the adjoint representation of Lie $H$-pseudoalgebra by $ad_L$.

Let $(L, {[\cdot * \cdot]}_L)$ be a Lie $H$-pseudoalgebra and let $M$ be an $L$-module with the action map $\rho$ defined by $x * u$ for $x \in L$, $u \in M$. The \emph{semidirect product} of $L$ and $M$, denoted by $L \ltimes M$, is the Lie $H$-pseudoalgebra defined on the direct sum $H$-module $L \oplus M$ with the pseudobracket given by
\begin{align*}
    {[(x,u) * (y, v)]}_{L \ltimes M} := \Big({[x * y]}_L, x * v - ((12) \otimes_H \mathrm{id}_M)(y * u)\Big),
\end{align*}
for all $(x,u), (y, v) \in L \oplus M$.
\subsection{Cohomology for Lie $H$-Pseudoalgebras}
Let $L$ be a Lie $H$-pseudoalgebra and $M$ be an $L$-module. For $n \geq 0$, the space of $n$-cochains $C^n(L, M)$ is defined~as
\begin{align*}
    C^n(L, M) := 
    \begin{cases}
        \mathbb{k} \otimes_H M & \text{if } n = 0, \\
        \mathrm{Lin}(\{L, \dots, L\}, M) = \mathrm{Hom}_{H^{\otimes n}}\left(L^{\otimes n}, H^{\otimes n} \otimes_H M\right) & \text{if } n \geq 1.
    \end{cases}
\end{align*}
The coboundary operator $d: C^n(L, M) \to C^{n+1}(L, M)$ is defined as follows.

For $n = 0$, let $1 \otimes_H u \in C^0(L, M)$. Then $d(1 \otimes_H u) \in C^1(L, M) = \mathrm{Hom}_H(L, M)$ is given by
\begin{align*}
    (d(1 \otimes_H u))(x) = \sum_i (\mathrm{id} \otimes \varepsilon)x*u.
\end{align*}

For $n \geq 1$, let $f \in C^n(L, M)$. Then $(df)(x_1, \dots, x_{n+1})$ is given by
\begin{align*}
    (df)(x_1, \dots, x_{n+1}) &= \sum_{i=1}^{n+1} (-1)^{i+1} (\sigma_{1\to i} \otimes_H \mathrm{id}_M)\Big( x_i * f(x_1, \dots, \hat{x}_i, \dots, x_{n+1}) \Big) \\&+ \sum_{1 \leq i < j \leq n+1} (-1)^{i+j} (\sigma_{1\to i, 2\to j} \otimes_H \mathrm{id}_M)f\Big( {[x_i * x_j]}_L, x_1,\cdots,\hat{x_i},\cdots \hat{x_j},\cdots x_{n+1}\Big),
\end{align*}
where $\sigma_{1\to i}$ is the permutation that moves the factor corresponding to $x_i$ to the first position in $H^{\otimes (n+1)}$, and $\sigma_{1\to i, 2\to j}$ is the permutation that moves the factors corresponding to $x_i$ and $x_j$ to the first two positions in $H^{\otimes (n+1)}$.
  The resulting complex $\Big(C^\bullet(L, M)= \bigoplus C^n(L, M), d\Big)$ defines the cohomology groups $H^\bullet(L, M)$. 

\subsection{Rota--Baxter Operators on Lie $H$-Pseudoalgebras}
We now recall Rota--Baxter operators on Lie $H$-pseudoalgebras.
\begin{definition} {\cite[Definition 2.5]{Liu-Wang-2020}}
Let $L$ be a Lie $H$-pseudoalgebra. A \emph{Rota--Baxter operator of weight $\lambda \in \mathbb{k}$} on $L$ is an $H$-linear map $T: L \to L$ that satisfies the \emph{Rota--Baxter identity} for all $x,y \in L$:
\begin{align*}
     {[T(x) * T(y)]}_L = (id_{H^{\otimes2}} \otimes_H T)\Big( {[T(x) * y]}_L + {[x * T(y)]}_L + \lambda {[x*y]}_L \Big).
\end{align*} 

A \emph{Rota--Baxter Lie $H$-pseudoalgebra} is a pair $(L, T)$ where $L$ is a Lie $H$-pseudoalgebra and $T: L \to L$ is a Rota--Baxter operator on $L$. We often denote such a structure simply by $L_T$.
\end{definition}

\begin{example} Let $\pi :\mathfrak{g}\to\mathfrak{b}$ be a homomorphism  of Lie algebras and $L$ be a Rota-Baxter operator  $T$ of weight $\lambda$ Lie $U(\mathfrak{g})$-pseudoalgebra. Then $U(\mathfrak{b})\otimes_{U(\mathfrak{g})}L$ is a Rota-Baxter $U(\mathfrak{b})$-pseudoalgebra with the Rota-Baxter operator $id_{U(\mathfrak{b})}\otimes_{U(\mathfrak{g})}T$ of weight $\lambda$. In particular, let $\mathfrak{g}=0$.
Then the current pseudoalgebra $\mathrm{Cur}(L) := U(\mathfrak{b}) \otimes L$ is a Rota-Baxter Lie $U(\mathfrak{b})$-pseudoalgebra. This construction is generalized in \cite[Example 2.9]{Liu-Wang-2020}.
\end{example}

A morphism between two Rota--Baxter Lie $H$-pseudoalgebras $(L, T)$ and $(L', T')$ is an $H$-module homomorphism $\phi: L \to L'$ such that
$(\mathrm{id}_{H^{\otimes 2}} \otimes_H \phi)\Big({[x * y]}_L\Big) = {[\phi(x) * \phi(y)]}_{L'} \quad \text{and} \quad  \phi \Big(T(x)\Big) = T'\Big(\phi(x)\Big),$
for all $x,y \in L$. An automorphism of $(L, T)$ is an invertible endomorphism of $L_T$.

\begin{proposition}\label{inducedrb}
Let $(L, {[\cdot * \cdot]}_L)$ be a Lie $H$-pseudoalgebra and $T: L \to L$ be a Rota--Baxter operator of weight $\lambda$ on $L$. Define a new pseudobracket ${[\cdot * \cdot]}_T$ on $L$ by
   $ {[x * y]}_T := {[T(x) * y]}_L + {[x * T(y)]}_L + \lambda {[x * y]}_L, \quad \text{for all } x, y \in L.$
Then:
\begin{enumerate}
    \item $(L, {[\cdot * \cdot]}_T)$ is a Lie $H$-pseudoalgebra.
    \item The map $T: (L, {[\cdot * \cdot]}_T) \to (L, {[\cdot * \cdot]}_T)$ is a Rota--Baxter operator of weight zero on the new Lie $H$-pseudoalgebra $(L, {[\cdot * \cdot]}_T)$.
     \item The map $T: (L, {[\cdot * \cdot]}_T) \to (L, {[\cdot * \cdot]}_L)$ is a morphism of Lie $H$-pseudoalgebras.
\end{enumerate}
\end{proposition}

\begin{proof}
We prove each statement in turn.
\begin{enumerate}
\item To show that $(L, {[\cdot * \cdot]}_T)$ is a Lie $H$-pseudoalgebra, we must verify skew-symmetry and the Jacobi identity.
\begin{enumerate}
    \item For any $x, y \in L$, we have
    \begin{align*}
        {[y * x]}_T &= {[T(y) * x]}_L + {[y * T(x)]}_L + \lambda {[y * x]}_L \\
        &= -((12)\otimes_H \mathrm{id}_L){[x * T(y)]}_L - ((12)\otimes_H \mathrm{id}_L){[T(x) * y]}_L - \lambda ((12)\otimes_H \mathrm{id}_L){[x * y]}_L \\
        &= -((12)\otimes_H \mathrm{id}_L)\Big( {[T(x) * y]}_L + {[x * T(y)]}_L + \lambda {[x * y]}_L \Big) \\
        &= -((12)\otimes_H \mathrm{id}_L){[x * y]}_T,
    \end{align*}
    where we used the skew-symmetry of the original bracket ${[\cdot * \cdot]}_L$.
    \item  The Jacobi identity for ${[\cdot * \cdot]}_T$ follows from Proposition 3.2 in \cite{Liu-Wang-2020} which shows that if $L$ is a Leibniz $H$-pseudoalgebra (a generalization of a Lie $H$-pseudoalgebra) and $T$ is a Rota--Baxter operator, then the new bracket defined as above also satisfies the Jacobi identity. Since every Lie $H$-pseudoalgebra is a Leibniz $H$-pseudoalgebra, the result holds.
\end{enumerate}
Therefore, $(L, {[\cdot * \cdot]}_T)$ is a Lie $H$-pseudoalgebra.
\item We need to show that $T$ is a Rota--Baxter operator of weight $\lambda$ on $(L, {[\cdot * \cdot]}_T)$. Consider that
\begin{align*}
    {[T(x) * T(y)]}_T = {[T(T(x)) * T(y)]}_L + {[T(x) * T(T(y))]}_L + \lambda ({[T(x) * T(y)]}_L).
\end{align*}
Computing each term individually, we get
\begin{align*} {[T(T(x)) * T(y)]}_L&=(\mathrm{id}_{H^{\otimes 2}} \otimes_H T)\Big({[T(T(x)) * y]}_L + {[T(x) * T(y)]}_L + \lambda {[T(x) * y]}_L\Big), \\
    {[T(x) * T(T(y))]}_L&= (\mathrm{id}_{H^{\otimes 2}} \otimes_H T)\Big({[T(x) * T(y)]}_L + {[x * T(T(y))]}_L + \lambda {[x * T(y)]}_L\Big),\\
     {[T(x) * T(y)]}_L  &=(\mathrm{id}_{H^{\otimes 2}} \otimes_H T) \Big({[T(x) * y]}_L + {[x * T(y)]}_L + \lambda {[x * y]}_L\Big).
\end{align*}
Adding these terms yields
\begin{align*}
    &(\mathrm{id}_{H^{\otimes 2}} \otimes_H T)\Big( {[T(x) * y]}_T + {[x * T(y)]}_T +\lambda  {[x * y]}_T\Big).
\end{align*}
Thus, we have  \begin{align*}
     {[T(x) * T(y)]}_T= &(\mathrm{id}_{H^{\otimes 2}} \otimes_H T)\Big( {[T(x) * y]}_T + {[x * T(y)]}_T +\lambda  {[x * y]}_T\Big).
\end{align*}
\item We need to show that $T$ preserves the pseudobracket, i.e., ${[T(x) * T(y)]}_L = (id_{H^{\otimes 2}} \otimes_H T)({[x * y]}_T)$ for all $x, y \in L$. By the definition of ${[\cdot * \cdot]}_T$, we have
\begin{align}
    {[x * y]}_T = {[T(x) * y]}_L + {[x * T(y)]}_L + \lambda {[x * y]}_L.
\end{align}
Applying $(id_{H^{\otimes 2}} \otimes_H T)$ to both sides gives
\begin{align*}
    (id_{H^{\otimes 2}} \otimes_H T)({[x * y]}_T) = (id_{H^{\otimes 2}} \otimes_H T)\Big( {[T(x) * y]}_L + {[x * T(y)]}_L + \lambda {[x * y]}_L \Big).
\end{align*}
Since $T$ is a Rota--Baxter operator of weight $\lambda$ on $(L, {[\cdot * \cdot]}_L)$, we obtain
\begin{align*}
    (id_{H^{\otimes 2}} \otimes_H T)({[x * y]}_T) = {[T(x) * T(y)]}_L.
\end{align*}
which shows that $T$ is a morphism of two Lie $H$-pseudoalgebras.
\end{enumerate}It completes the proof.
 \end{proof} We call $(L, {[\cdot * \cdot]}_T)$ the algebra as ``deformed Lie $H$-pseudoalgebra" and its corresponding ``deformed Rota--Baxter Lie $H$-pseudoalgebra" is denoted by $(L, {[\cdot * \cdot]}_T, T)$ or  $(L_T, {[\cdot * \cdot]}_T)$. 

\begin{definition}\label{rbrep}
A representation of a Rota--Baxter Lie $H$-pseudoalgebra $(L, {[\cdot*\cdot]}_L, T)$ of weight $\lambda \in \mathbb{k}$ is a triple $(M, S, \rho)$ where $M$ is an $L$-module, $S: M \to M$ is an $H$-linear map $\rho$ is the action map defined by $\rho(x, u)=x*u$ such that for all $x \in L$, $u \in M$,
\begin{align*}
       T(x) * S(u) = (\mathrm{id}_{H^{\otimes 2}} \otimes_H S)\Big( T(x) * u + x * S(u) + \lambda (x * u )\Big).
\end{align*}
\end{definition}This condition ensures that the semidirect product $L \ltimes M$ can be endowed with a Rota--Baxter operator $T \oplus S$, see Proposition \ref{propsemidirect}.
\begin{example}
Consider a Rota--Baxter Lie $H$-pseudoalgebra $(L, {[\cdot * \cdot]}_L, T)$, then
\begin{enumerate}
    \item The triple $(L, {[\cdot * \cdot]}_L, T)$ is a representation of itself, where the action $\rho$ is given by the pseudobracket ${[\cdot * \cdot]}_L$. This is called the adjoint representation.
    \item For any representation $(M, \rho)$ of the underlying Lie $H$-pseudoalgebra $L$, the triple $(M, \rho, 0)$ is a representation of the Rota--Baxter Lie $H$-pseudoalgebra $(L, T)$.
\end{enumerate}

Let $(L, T)$ be a Rota--Baxter Lie $H$-pseudoalgebra of weight $\lambda$ and let $(M, \pm S)$ be a representation of it. Then $(M, \pm S-\lambda \mathrm{id}_M)$ is also a representation of Rota--Baxter Lie $H$-pseudoalgebra.
\end{example} 

Similar to Definition~\ref{rbrep}, we can define a representation on the \emph{deformed Rota--Baxter Lie $H$-pseudoalgebra}.
\begin{proposition}\label{inducedrbrep}
Let $(L, {[\cdot * \cdot]}_L, T)$ be a Rota--Baxter Lie $H$-pseudoalgebra of weight $\lambda$ and let $(M, S, \rho)$ be a representation on it. Define a new action $\rho^T: L \otimes M \to H^{\otimes 2} \otimes_H M$ defined by \begin{align*}
\rho^T(x, u) =x *_T u := T(x) * u + x * S(u) + \lambda (x * u),
\end{align*}
for all $x \in L$, $u \in M$. Then $(M, S, \rho^T)$ is also a representation of the deformed Rota--Baxter Lie $H$-pseudoalgebra  $(L, {[\cdot * \cdot]}_{T}, T)$  with respect to this new action $*_T$.
\end{proposition}
\begin{proof}
To show that $(M, S, \rho^T)$ is a representation of the deformed Rota--Baxter Lie $H$-pseudoalgebra $(L, {[\cdot * \cdot]}_{T}, T)$ with the new action $*_T$, we need to verify the following equations for all $x \in L$, $m \in M$,
\begin{align*}
  ( {[x * y]}_T) *_T u  &=  x *_T (y *_T u) - ((12)\otimes_H \mathrm{id}_M)(y *_T (x *_T u)),\\
   T(x) *_T S(u) &= (\mathrm{id}_{H^{\otimes 2}} \otimes_H S)\Big( T(x) *_T u + x *_T S(u) + \lambda (x *_T u) \Big).
\end{align*}
This follows directly from the definition of the new action $*_T$ and the fact that $(M, S)$ is a representation of $L_T$. The details are omitted due to space constraints.
\end{proof}
\begin{proposition}\label{propsemidirect}
Let $(L_T, {[\cdot * \cdot]}_L)$ be a Rota--Baxter Lie $H$-pseudoalgebra of weight $\lambda$ and let $(M, S)$ be a representation of $L_T$. Then the semidirect product Lie $H$-pseudoalgebra $L \ltimes M$ can be endowed with a Rota--Baxter operator $T \oplus S$ of weight $\lambda$, defined by
\begin{align*}
     (T \oplus S)(x, u) := \Big(T(x), S(u)\Big),\qquad \forall \quad (x,u) \in L \oplus M.
\end{align*}
\end{proposition}
\begin{proof}
To show that $T \oplus S$ is a Rota--Baxter operator of weight $\lambda$ on $L \ltimes M$, we need to verify that for all $(x, u), (y, v) \in L \oplus M$,
\begin{align*}
     &{\Big[(T \oplus S)(x, u) * (T \oplus S)(y, v)\Big]}_{L \ltimes M}\\ &= (\mathrm{id}_{H^{\otimes 2}} \otimes_H (T \oplus S))\Big( {[(T \oplus S)(x, u) * (y, v)]}_{L \ltimes M} + {[(x, u) * (T \oplus S)(y, v)]}_{L \ltimes M} + \lambda {[(x, u) * (y, v)]}_{L \ltimes M} \Big).
\end{align*}
The left-hand side is
\begin{align}
    {\Big[(T(x), S(u)) * (T(y), S(v))\Big]}_{L \ltimes M} = \Big({[T(x) * T(y)]}_L, T(x) * S(v) - ((12)\otimes_H \mathrm{id}_M)(T(y) * S(u))\Big).
\end{align}
For the right-hand side, we first compute
\begin{align*}
    {\Big[(T(x), S(u)) * (y, v)\Big]}_{L \ltimes M} &= \Big({[T(x) * y]}_L, T(x) * v - ((12)\otimes_H \mathrm{id}_M)(y * S(u))\Big), \\
    {\Big[(x, u) * (T(y), S(v))\Big]}_{L \ltimes M} &= \Big({[x * T(y)]}_L, x * S(v) - ((12)\otimes_H \mathrm{id}_M)(T(y) * u)\Big), \\
    \lambda {\Big[(x, u) * (y, v)\Big]}_{L \ltimes M} &= \lambda \Big({[x * y]}_L, x * v - ((12)\otimes_H \mathrm{id}_M)(y * u)\Big).
\end{align*}
Adding these together and applying $(\mathrm{id}_{H^{\otimes 2}} \otimes_H (T \oplus S))$ gives
\begin{align*}
      &(\mathrm{id}_{H^{\otimes 2}} \otimes_H T)\Big( {[T(x) * y]}_L + {[x * T(y)]}_L + \lambda {[x * y]}_L, \\& T(x) * v + x * S(v) + \lambda x * v - ((12)\otimes_H \mathrm{id}_M)(y * S(u) + T(y) * u + \lambda y * u) \Big)\\&=(\mathrm{id}_{H^{\otimes 2}} \otimes_H T)\Big( {[x * y]}_T,  x*_T v - ((12)\otimes_H \mathrm{id}_M)(y * _T u) \Big)\\&=\Big( (\mathrm{id}_{H^{\otimes 2}} \otimes_H T){[x * y]}_T, (\mathrm{id}_{H^{\otimes 2}} \otimes_H S) (x*_T v - ((12)\otimes_H \mathrm{id}_M)y * _T u) \Big)
      \\&=\Big({[T(x) * T (y)]}_L, T(x) *_T S(v) - ((12)\otimes_H \mathrm{id}_M)(T(y) *_T S(u)) \Big).
\end{align*}This completes the proof.
\end{proof}

\section{Cohomology of Rota--Baxter Lie $H$-Pseudoalgebras}\label{sec:cohomology}
\label{sec:RB-cohomology}
In this section, we develop a cohomology of  $\lambda$-weighted relative Rota--Baxter Lie $H$-pseudoalgebras, following the Maurer--Cartan approach introduced in \cite{Das2022}. Let $L$ and $M$ be Lie $H$-pseudoalgebras, and suppose $M$ is a module over $L$ by an action $\rho$ defined by $[x * u]$ for $x \in L$, $u \in M$. We assume that the action $\rho$ is given by derivations, i.e., for all $x \in L$ and $u, v \in M$,
\begin{align}\label{eq:derivation-condition}
    \rho(x, {[u * v]}_M )= {[\rho(x, u) * v]}_M + ((12)\otimes_Hid){[u * \rho(x, v)]}_M.
\end{align}This ensures that $(L, M)$ forms a matched pair of Lie $H$-pseudoalgebras.
A $\lambda$-weighted relative Rota--Baxter operator $T : M \to L$ is an $H$-linear map satisfying the identity
  $ { [T(u) * T(v)]}_L = (\mathrm{id}_{H^{\otimes 2}} \otimes_H T)\big(\rho(T(u),v )- \rho(T(v),u) + \lambda {[u * v]}_M \big), \quad \forall u,v \in M,$
where $\rho :L \otimes M \to H^{\otimes 2} \otimes_H M$ is the representation map encoding the $L$-module structure on $M$. 
Our goal is to define a cohomology theory for such an operator $T$ that gives infinitesimal and formal deformations. We proceed in two steps.  

\begin{enumerate}
    \item First is to construct a differential graded Lie algebra (dgLA) whose Maurer--Cartan elements are precisely $\lambda$-weighted relative Rota--Baxter operators.
    \item Second is to use the twisting procedure to obtain a cochain complex $(C^\bullet_T(M,L), d_T)$, whose cohomology $H^\bullet_T(M,L)$ controls the deformation theory of $T$.
\end{enumerate}
This framework generalizes the Chevalley–Eilenberg cohomology of Lie $H$-pseudoalgebras.
\subsection{Maurer--Cartan Characterization and Cohomology}
\label{subsec:MC-characterization}
In this subsection, we construct a differential graded Lie algebra (dgLA) whose Maurer--Cartan elements are precisely $\lambda$-weighted relative Rota--Baxter operators of Lie $H$-pseudoalgebras. Using this characterization, we define the cohomology of such an operator. Recall that a Maurer--Cartan element in a differential graded Lie algebra (dgLA) is defined as follows.
 Let $(L = \bigoplus_{k=0}^{\infty}
L_k, [\cdot,\cdot], d) $ be a differential graded Lie algebra. A degree $1$ element $\theta\in L_1$ is called a Maurer--Cartan element of $L$ if it satisfies the Maurer--Cartan
equation $d\theta +\frac{1}{2}
[\theta,\theta ] = 0.$
A graded Lie algebra is a differential graded Lie algebra with $d = 0.$ Then we have

\begin{proposition}  Let  $(L = \bigoplus_{k=0}^{\infty}
L_k, [\cdot,\cdot]) $ be a graded Lie algebra and let $\theta\in L_1$ be a Maurer--Cartan element. Then the Maurer--Cartan element $\theta$ induces a  map $d_\theta : L \to L,$  defined by $d_\theta (u) := [\theta, u], \quad \forall u \in L.$ We call it
  differential on $L$. For any $v \in L_1,$ the sum $\theta + v $ is a Maurer--Cartan element of the graded Lie algebra $(L, [\cdot,\cdot])$ if and only if $v $ is a Maurer--Cartan element of the differential graded Lie algebra $(L, [\cdot,\cdot], d_\theta)$.
\end{proposition} 
Let $(M, \rho)$ be a representation of a Lie $H$-pseudoalgebra $L$, where $\rho:L \otimes M \to H^{\otimes 2} \otimes_H M$ denotes the action map. Consider the graded space $  C^\bullet(M, L) := \bigoplus_{k=0}^{\infty} Hom_{H^{\otimes k}}(M^{\boxtimes k}, H^{\otimes k} \otimes_H L),$ 
with the convention that $M^{\boxtimes 0} = \mathbb{k}$ and the $0$-cochains are given by $C^0(M, L) = L$. This space carries a degree $0$ skew-symmetric bracket operation
\[ \llbracket \cdot, \cdot \rrbracket : Hom_{H^{\otimes m}}(M^{\boxtimes m}, H^{\otimes m} \otimes_H L) \times Hom_{H^{\otimes n}}(M^{\boxtimes n}, H^{\otimes n} \otimes_H L) \longrightarrow Hom_{H^{\otimes (m+n)}}(M^{\boxtimes (m+n)}, H^{\otimes (m+n)} \otimes_H L)
\]
defined by
\begin{equation}
    \begin{aligned}\label{derivedbracket}
    \llbracket P, Q \rrbracket &(u_1, \dots, u_{m+n}) = \\
    &\sum_{\sigma \in S(n,1,m-1)} (-1)^\sigma (\sigma \otimes_H P)\Big( \rho\big(Q(u_{\sigma(1)}, \dots, u_{\sigma(n)}), u_{\sigma(n+1)}\big), u_{\sigma(n+2)}, \dots, u_{\sigma(n+m)} \Big) \\
    &- (-1)^{mn} \sum_{\sigma \in S(m,1,n-1)} (-1)^\sigma (\sigma \otimes_H Q) \Big( \rho\big(P(u_{\sigma(1)}, \dots, u_{\sigma(m)}), u_{\sigma(m+1)}\big), u_{\sigma(m+2)}, \dots, u_{\sigma(m+n)} \Big) \\
    &+ (-1)^{mn} \sum_{\sigma \in S(m,n)} (-1)^\sigma (\sigma\otimes_Hid){\big[P(u_{\sigma(1)}, \dots, u_{\sigma(m)})* Q(u_{\sigma(m+1)}, \dots, u_{\sigma(m+n)})\big]}_L,
    \end{aligned}
\end{equation}
for all $P \in Hom_{H^{\otimes m}}(M^{\boxtimes m}, H^{\otimes m} \otimes_H L)$, $Q \in Hom_{H^{\otimes n}}(M^{\boxtimes n}, H^{\otimes n} \otimes_H L)$, and $u_1,\dots,u_{m+n} \in M$. Here $S(n,m)$ denotes the set of $(n,m)$-unshuffles, that is permutations $\sigma$ of $\{1,\dots,m+n\}$ such that
    \[  \sigma(1) < \cdots < \sigma(n), \quad \sigma(n+1) < \cdots < \sigma(n+m);
    \] and $S(n,1,m-1)$ is the set of permutations $\sigma$ of $\{1,\dots,n+m\}$ satisfying
    \[
        \sigma(1) < \cdots < \sigma(n), \quad \sigma(n+1), \quad \sigma(n+2) < \cdots < \sigma(n+m).
    \] whereas $(-1)^\sigma$ is the Koszul sign associated to the permutation $\sigma$, determined by the grading in the pseudotensor category $M^*(H)$.
    
   Note that for all $x,y \in L$, we identify ${[x * y]}_L$ with the image of the pseudobracket $\pi_L(x,y) \in H^{\otimes 2} \otimes_H L$.
   Furthermore, we have the following key result.
\begin{theorem}\label{thm:MC-Rota--Baxter}
    With the above notations, $(C^\bullet(M, L), \llbracket \cdot, \cdot \rrbracket)$ is a graded Lie algebra. Its Maurer--Cartan elements of degree $1$ are precisely the $\lambda$-weighted relative Rota--Baxter operators $T: M \to L$. When $\lambda = 0$, it corresponds to $\mathcal{O}$-operators or (relative Rota--Baxter operators of weight $0$) on $L$ with respect to the representation $(M, \rho)$.
\end{theorem}
\begin{proof}
The structure arises via the derived bracket construction associated with the Nijenhuis--Richardson bracket in the category of Lie $H$-pseudoalgebras. Let $W = L \oplus M$ be the direct sum $H$-module. Consider the graded space \begin{align*}
    \mathfrak{g} = \bigoplus_{k=0}^{\dim W} Hom_{H^{\otimes k}}(W^{\boxtimes k}, H^{\otimes k} \otimes_H W),
\end{align*} equipped with the Nijenhuis--Richardson bracket ${[\cdot,\cdot]}_{NR}$ of degree $-1$, defined by \begin{align*}
    {[f,g]}_{NR} = f \diamond g - (-1)^{(m-1)(n-1)} (\mathrm{id}_{H^{\otimes m}} \otimes_H g) \diamond f,
\end{align*} defined by
 {\begin{align*}
      (f \diamond g)(w_1, \dots, w_{m+n-1}) = \sum_{\sigma \in S(n,m-1)} (-1)^\sigma (\sigma\otimes_Hid)f\left(g(w_{\sigma(1)}, \dots, w_{\sigma(n)}), w_{\sigma(n+1)}, \dots, w_{\sigma(m+n-1)}\right).
  \end{align*}}
  for $f \in Hom_{H^{\otimes m}}(W^{\boxtimes m}, H^{\otimes m} \otimes_H W)$, $g \in Hom_{H^{\otimes n}}(W^{\boxtimes n}, H^{\otimes n} \otimes_H W)$.  Denote by $\pi_L \in  Hom_{H^{\otimes 2}}(L^{\boxtimes 2}, H^{\otimes 2} \otimes_H L)$ the pseudobracket on $L$, and let $\rho \in Hom_{H^{\otimes 2}}(L \boxtimes M, H^{\otimes 2} \otimes_H M)$ encode the action of $L$ on $M$. Then the element $\mu_0 = \pi_L + \rho \in \mathfrak{g}_1$ satisfies the Maurer--Cartan equation $ {[\mu_0, \mu_0]}_{NR} = 0$ if and only if $L$ is a Lie $H$-pseudoalgebra and $M$ is an $L$-module. Moreover, under the derivation condition \eqref{eq:derivation-condition}, the action $\rho$ is by derivations of the Lie $H$-pseudoalgebra structure on $M$. This implies that the Nijenhuis--Richardson bracket between $\mu_0 = \pi_L + \rho$ and $\pi_M$ vanishes:
\begin{align*}
   { [\mu_0, \pi_M]}_{NR} ={ [\pi_L, \pi_M]}_{NR} + {[\rho, \pi_M]}_{NR} = 0.
\end{align*}
Indeed, ${[\pi_L, \pi_M]}_{NR} = 0$ because $\pi_L$ and $\pi_M$ act on disjoint components ($L$ and $M$, respectively), and ${[\rho, \pi_M]}_{NR} = 0$ precisely encodes the derivation property \eqref{eq:derivation-condition}. Consequently, the elements $\mu_0$ and $-\lambda \pi_M$ commute in $(\mathfrak{g}, {[\cdot,\cdot]}_{NR})$, and the operator $d := {[-\lambda \pi_M, \cdot]}_{NR}$ defines a differential that preserves the subspace $\mathfrak{a}$ and is compatible with the derived bracket $\llbracket \cdot, \cdot \rrbracket$.

This induces a square-zero differential $d_{\mu_0} = {[\mu_0, -]}_{NR}$ on $\mathfrak{g}$. Now consider the subspace
$$\mathfrak{a} = \bigoplus_{k=0}^{\infty} Hom_{H^{\otimes k}}(M^{\boxtimes k}, H^{\otimes k} \otimes_H L) \subset \mathfrak{g}.$$
It is straightforward to verify that $\mathfrak{a}$ is an abelian subalgebra, i.e., ${[f,g]}_{NR} = 0$ for all $f,g \in \mathfrak{a}$, due to disjoint domains. By the derived bracket construction by Voronov in \cite{Voronov2005}, the shifted space $\mathfrak{a}[-1]$ carries a graded Lie bracket defined by
 \begin{align*}
      \llbracket P, Q \rrbracket := (-1)^m {[{[\mu_0, P]}_{NR}, Q]}_{NR}, \quad P \in \mathfrak{a}[-1]^m,\ Q \in \mathfrak{a}[-1]^n.
 \end{align*}
This coincides exactly with the derived bracket in Eq. \eqref{derivedbracket}. Therefore, $(C^\bullet(M,L), \llbracket \cdot,\cdot \rrbracket)$ is a graded Lie algebra.

Next, we introduce the weight $\lambda\in \mathbb{k}$  by the element $-\lambda \pi_M \in Hom_{H^{\otimes 2}}(M^{\boxtimes 2}, H^{\otimes 2} \otimes_H M)$. This is also a Maurer--Cartan element, because ${[\pi_M, \pi_M]}_{NR} = 0$ 
. It induces a differential by\begin{align*} d_{-\lambda \pi_M} := {[-\lambda \pi_M,\cdot]}_{NR}.
\end{align*} This differential preserves the subspace $\mathfrak{a}$, since the bracket ${[-\lambda \pi_M, f]}_{NR}$ remains in $Hom(-, H^{\otimes \bullet} \otimes_H L)$ for $f \in \mathfrak{a}$. We denote the restriction of this differential to $\mathfrak{a}$ by $ d : \mathfrak{a} \to \mathfrak{a},$ and its explicit action on $f \in Hom_{H^{\otimes n}}(M^{\boxtimes n}, H^{\otimes n} \otimes_H L)$ is given by
\begin{align}
    (df)(u_1, \dots, u_{n+1}) = (-1)^{n-1} \sum_{i=1}^{n} (-1)^{i-1} f(u_1, \dots, u_{i-1}, \lambda ({[u_i * u_{i+1}]}_M), u_{i+2}, \dots, u_{n+1}).
\end{align}
In particular, if $n=1$, we have $(df)(u_1,  u_{2}) =  f(  \lambda ({[u_1 * u_{2}]}_M)) $  

Finally, the elements $\mu_0=\pi_L + \rho$ and $-\lambda \pi_M$ satisfy the 
compatibility condition $ {[\pi_L + \rho, -\lambda \pi_M]}_{NR} = 0$. This holds because their domains do not interact nontrivially, i.e., $\pi_L + \rho$ acts between $L$ and $M$, while $-\lambda \pi_M$ acts entirely within $M$. As a result, the differential $d$ is a graded derivation for the derived bracket $\llbracket \cdot,\cdot \rrbracket$, meaning that for homogeneous $P, Q \in \mathfrak{a}[-1]$,
\begin{align*}
     d\llbracket P, Q\rrbracket= \llbracket dP, Q\rrbracket+ (-1)^m \llbracket P, dQ\rrbracket, \quad \text{for } P \in \mathfrak{a}[-1]^m.
\end{align*}
Therefore, the triple $   \left( \mathfrak{a}[-1], \llbracket \cdot,\cdot\rrbracket, d \right)$ forms a differential graded Lie algebra (dgLA).

Now, let $T \in Hom_H (M,L) $ be a degree $1$ element, then derived bracket in Eq. \eqref{derivedbracket} is given by
\begin{equation}\begin{aligned}\label{derivedbracketrota}\llbracket T, T \rrbracket (u_1, u_2) 
    &= \sum_{\sigma \in S(1,1)} (-1)^\sigma (\sigma \otimes_H T) \Big( \rho\big(T(u_{\sigma(1)}), u_{\sigma(2)}\big) \Big) \\
    &\quad + \sum_{\sigma \in S(1,1)} (-1)^\sigma (\sigma \otimes_H T) \Big( \rho\big(T(u_{\sigma(1)}), u_{\sigma(2)}\big) \Big) \\
    &\quad - \sum_{\sigma \in S(1,1)} (-1)^\sigma (\sigma\otimes_Hid)\big{[T(u_{\sigma(1)}) * T(u_{\sigma(2)})\big]}_L
    \\ &= 2\Big((\mathrm{id}_{H^{\otimes 2}} \otimes_H T)(\rho(T(u_1), u_2)) - ((12) \otimes_H T)(\rho(T(u_2), u_1))\Big) \\
&\quad - \Big({[T(u_1) * T(u_2)]}_L - ((12)\otimes_Hid){[T(u_2) * T(u_1)]}_L\Big)
\end{aligned}
\end{equation}

and $d_{\mu_0} T $ after restriction to subspace $\mathfrak{a}$ is given by 
\begin{align}\label{restricteddifferentialrota}
   d_{-\l\pi_M} T (u_1, u_2)={[ -\l\pi_M, T]}_{NR}(u_1, u_2) = (\mathrm{id}_{H^{\otimes 2}} \otimes_H T)(\lambda {[u_1 * u_{2}]}_M) .
\end{align}
Then $T$ is a Maurer--Cartan element if it satisfies
\begin{align*}
 d_{ -\lambda\pi_M} T+ \frac{1}{2}\llbracket T, T \rrbracket(u_1,u_2) &=(\mathrm{id}_{H^{\otimes 2}} \otimes_H T)(\lambda {[u_1 *u_2]}_M) \\&+\Big( (\mathrm{id}_{H^{\otimes 2}} \otimes_H T)\rho(T(u_1),u_2) - ((12)\otimes_H T)\rho(T(u_2),u_1 )- {[T(u_1)* T(u_2)]}_L \Big) =0.
\end{align*} 
This is precisely a weighted relative Rota--Baxter operator. Thus, the Maurer--Cartan elements in $(C^\bullet(M, L), \llbracket \cdot,\cdot \rrbracket)$ are exactly the weighted relative Rota--Baxter operator $T: M \to L$ with respect to the representation $(M, \rho)$.
\end{proof}
\subsection{Cohomology of $\lambda$-Weighted Relative Rota--Baxter Operators}Given a $\lambda$-weighted relative Rota--Baxter operator $T : M \to L$, Theorem~\ref{thm:MC-Rota--Baxter} establishes that $T$ corresponds to a Maurer--Cartan element in the differential graded Lie algebra $(C^\bullet(M,L), \llbracket \cdot, \cdot \rrbracket, d)$. This allows us to twist the dgLA by $T$ and define its cohomology.
\begin{definition}
The \textbf{cochain complex of $T$} is defined on the graded space
\[ C^\bullet_T(M, L) := \bigoplus_{n \geq 0} Hom_{H^{\otimes n}}(M^{\boxtimes n}, H^{\otimes n} \otimes_H L),
\] with $C^0_T(M, L) = L$, and differential
  $d_T := d + \llbracket T,.\rrbracket.$ 
\end{definition}
Since $T$ satisfies the Maurer--Cartan equation, we have $d_T^2 = 0$, that yields  a cochain complex $(C^\bullet_T(M, L), d_T)$.
\begin{definition}
The \textbf{$n$-th cohomology group} of the $\lambda$-weighted relative Rota--Baxter operator $T$ is
\[
    H^n_T(M, L) := \frac{Z^n_T(M, L)}{B^n_T(M, L)}, \quad \text{where }
    Z^n_T(M, L) = ker(d_T), \quad B^n_T(M, L) = im(d_T).
\]
The total cohomology is denoted by $H^\bullet_T(M, L)$.
\end{definition}
Moreover, the twisted differential $d_T$ upgrades the graded Lie algebra $(\mathfrak{a}[-1], \llbracket\cdot,\cdot\rrbracket)$ to a new differential graded Lie algebra $(\mathfrak{a}[-1], \llbracket\cdot,\cdot\rrbracket, d_T)$.
\begin{theorem}
\label{thm:deformation-MC}
Let $T : M \to L$ be a $\lambda$-weighted relative Rota--Baxter operator. For any linear map $T' : M \to L$, the sum $T + T'$ is also a $\lambda$-weighted relative Rota--Baxter operator if and only if $T'$ is a Maurer--Cartan element in the twisted dgLA $(\mathfrak{a}[-1], \llbracket \cdot,\cdot\rrbracket, d_T)$, i.e.,
\begin{align*}
     d_T T' + \frac{1}{2} \llbracket T', T'\rrbracket = 0.
\end{align*}
\end{theorem}

\begin{proof}
By Theorem~\ref{thm:MC-Rota--Baxter}, $T + T'$ is a Rota Baxter operator if it is a Maurer Cartan element in the twisted grade Lie algebra $(\mathfrak{a}[-1], \llbracket \cdot,\cdot\rrbracket, d)$. We compute
\begin{align*}
    d(T + T') + \frac{1}{2} \llbracket T + T', T + T'\rrbracket 
    &= \underbrace{dT + \frac{1}{2} \llbracket T, T\rrbracket}_{=0} + \left( dT' + \llbracket T, T'\rrbracket \right) + \frac{1}{2} \llbracket T', T'\rrbracket \\
    &= d_T T' + \frac{1}{2} \llbracket T', T'\rrbracket.
\end{align*}It implies that $ d_T T' + \frac{1}{2} \llbracket T', T'\rrbracket=0$ and
 the proof is complete.
\end{proof} In the following, we will show that the cohomology of 
Rota--Baxter operator $T$
on Lie $H$-pseudoalgebra, that can be described in terms of Chevalley--Eilenberg cohomology of Lie  $H$-pseudoalgebras. We define a pseudomap $\rho_T : M \otimes L\to  id_{H^{\otimes 2}} \otimes_H L $  by
\begin{align*}
  \rho_T(u, x)  = \rho_T(u)(x):= (id_{H^{\otimes2}}\otimes_H T)(\rho(x)u) + {[T u * x]}_L  \quad \text{for } u \in M, \; x \in L,
\end{align*}
where the pseudobracket $[- * -]$ incorporates the $H$-action.

 \begin{proposition}
The map $\rho_T : M \otimes L \to  id_{H^{\otimes 2}} \otimes_H L $   defines a representation of the Lie $H$-pseudoalgebra $(M, {[\cdot * \cdot]}_T)$ on the $H$-module $L$.
\end{proposition}
\begin{proof}Assume that $\rho(x)u=x*u\in   {H^{\otimes 2}} \otimes_H M $. We extend the $L$-module structure to $H^{\otimes 2} \otimes_H L$ by setting
\begin{align*}
    {[a * (f \otimes g \otimes_H x)]}_L :=(id\otimes (f \otimes g)\Delta\otimes_Hid)  {[a * x]}_L
\quad \text{for all } a, x \in L,\ f, g \in H.
\end{align*}
Then
\begin{align*}
&{[\rho_T(u)* \rho_T(v)]}_L(x) 
\\&= \rho_T(u)  (\rho_T(v)x) - (((12) \otimes_H \mathrm{id}_L)(\rho_T(v)  (\rho_T(u) x)) \\
&= \rho_T(u)  \left((id_{H^{\otimes 2}}\otimes_H T)(x * v) + {[T v * x]}_L\right) \\
&\quad - ((123)\otimes_Hid))\left(\rho_T(v)  \left((12)\otimes_H T)(x * u) + {[T u * x]}_L\right)\right) \\
&= ((13)\otimes_H T)\left(((id_{H^{\otimes 2}}\otimes_H T)(x * v)) * u\right)  + {[T u * ((12)\otimes_H T)(x * v)]}_L + ((132)\otimes_H T)\left(({[T v * x]}_L) * u\right) \\
&\quad+ {[T u * {[T v * x]}_L]}_L 
- ((123)\otimes_H T)\left((id\otimes_H T)(x * u)) * v\right)  - ((132) \otimes_H \mathrm{id}_L){[T v * (id_{H^{\otimes 2}}\otimes_H T)(x * u)]}_L  \\
&\quad-(id_{H^{\otimes 3}}\otimes_H T)\left((23)\otimes_Hid){[T u * x]}_L * v\right) -((\sigma  \otimes_H \mathrm{id}_L) {[T v * {[T u * x]}_L]}_L.
\end{align*}
Next, applying the Rota--Baxter identity:
\begin{align*}
{\Big[T u * (id_{H^{\otimes 2}}\otimes_H T)(x * v)\Big]}_L &= (id_{H^{\otimes 2}}\otimes_H T)\Big({[T u * (x * v)]}_L + {[u * (id_{H^{\otimes 2}}\otimes_H T)(x * v)]}_L + \lambda {[u * (x * v)]}_L\Big), \\
{\Big[T v * (id_{H^{\otimes 2}}\otimes_H T)(x * u)\Big]}_L &= (id_{H^{\otimes 2}}\otimes_H T)\Big({[T v * (x * u)]}_L + {[v * (id_{H^{\otimes 2}}\otimes_H T)(x * u)]}_L + \lambda {[v * (x * u)]}_L\Big).
\end{align*}
Substituting these:
\begin{align*}
&= (id_{H^{\otimes 3}}\otimes_H T)\left(((id_{H^{\otimes 2}}\otimes_H T)(x * v)) * u\right)  \\
&\quad+ ((123)\otimes_H T)\left({[T u * (x * v)]}_L + {[u * (id_{H^{\otimes 2}}\otimes_H T)(x * v)]}_L + \lambda {[u * (x * v)]}_L\right) \\
&\quad + (id\otimes_H T)\left(((123)\otimes_Hid)({[T v * x]}_L) * u\right) + ((132)\otimes_Hid){[T u * {[T v * x]}_L]}_L \\
&\quad -((23)\otimes_H T)\left(((id\otimes_H T)(x * u)) * v\right) \\&\quad -(((12)\otimes_H T)\left({[T v * (x * u)]}_L + {[v * (id_{H^{\otimes 2}}\otimes_H T)(x * u)]}_L + \lambda {[v * (x * u)]}_L\right)  \\
&\quad - ((123)\otimes_H T)\left(({[T u * x]}_L) * v\right) -((132)\otimes_H \mathrm{id}_L){[T v * {[T u * x]}_L]}_L.
\end{align*}

Using the Lie $H$-pseudoalgebra Jacobi identity:
\[
{[T u * {[T v * x]}_L]}_L - ((\sigma \otimes \mathrm{id}_H) \otimes_H \mathrm{id}_L) {[T v * {[T u * x]}_L]}_L = {[{[T u * T v]}_L * x]}_L,
\]
and the Rota--Baxter identity on ${[T u * T v]}_L$:
\[
{[T u * T v]}_L = (id_{H^{\otimes 2}}\otimes_H T)\left({[T u * v]}_L+ {[u * T v]}_L+ \lambda {[u * v]}_L\right),
\]
we obtain
\[
{[T u * {[T v * x]}_L]}_L - ((\sigma \otimes \mathrm{id}_H) \otimes_H \mathrm{id}_L){[T v * {[T u * x]}_L]}_L = {[(id_{H^{\otimes 2}}\otimes_H T)\left({[T u * v]}_L+ {[u * T v]}_L+ \lambda {[u * v]}_L\right) * x]}_L.
\]

Now reorganize the $T$-terms:
\begin{align*}
&(id\otimes_H T)\Big((id\otimes_H T)(x * v)) * u + ((123)\otimes_Hid)({[T u * (x * v)]}_L+ {[u * (id\otimes_H T)(x * v)]}_L)\\&+((12)\otimes_Hid) ({[T v * x]}_L) * u\Big) \\
- &(id\otimes_H T)\Big((((23)\otimes_H T)(x * u)) * v + ((12)\otimes_Hid){[T v * (x * u)]}_L+ {[v * (id_{H^{\otimes 2}}\otimes_H T)(x * u)]}_L\\&+((123)\otimes_Hid) ({[T u * x]}_L) * v\Big) \\
+ &\lambda (((123)\otimes_H T)\Big({[u * (x * v)]}_L- (((12) \otimes_H \mathrm{id}_L){[v * (x * u)]}_L\Big) \\&+((132)\otimes_Hid) {[(id\otimes_H T)\Big({[T u * v]}_L+ {[u * T v]}_L+ \lambda {[u * v]}_L\Big) * x]}_L.
\end{align*}

Using module compatibility and $H$-linearity, 
we get
\begin{align*}
&= (id\otimes_H T)\Big(x * ({[T u * v]}_L  + {[u * T v]}_L + \lambda {[u * v]}_L)\Big) + ((132)\otimes id){\Big[(id\otimes_H T)\Big({[T u * v]}_L + {[u * T v]}_L + \lambda {[u * v]}_L\Big) * x\Big]}_L \\
&=((132)\otimes_Hid)( \rho_T({[T u * v]}_L + {[u * T v]}_L + \lambda {[u * v]}_L)(x) )\\
&=((132)\otimes_Hid) \rho_T({[u * v]}_T)(x).
\end{align*}This completes the proof.
\end{proof}
Let $T : M \to L$ be a $\lambda$-weighted relative Rota--Baxter operator between Lie $H$-pseudoalgebras. Then one may consider the Chevalley--Eilenberg cohomology of the Lie $H$-pseudoalgebra $(M, {[\cdot * \cdot]}_T)$ with coefficients in the representation $(L, \rho_T)$. The $n$-th cochain group is 
\begin{align}
    C^n_{\mathrm{CE}}(M, L) = \mathrm{Hom}_{H^{\otimes n}}(M^{\boxtimes n}, H^{\otimes n} \otimes_H L), \quad \text{for } n \geq 0,
\end{align}
and the coboundary map $\delta_{\mathrm{CE}} : C^n_{\mathrm{CE}}(M, L) \to C^{n+1}_{\mathrm{CE}}(M, L)$ is given by the standard Chevalley--Eilenberg differential adapted to pseudomaps: 
\begin{align*}
(\delta_{\mathrm{CE}} f)(u_1, \dots, u_{n+1}) 
&= \sum_{i=1}^{n+1} (-1)^{i+1} ((\sigma_{1,i}\otimes_Hid)\rho_T(u_i) f(u_1, \dots, \widehat{u_i}, \dots, u_{n+1}) \\
&\quad + \sum_{1 \leq i < j \leq n+1} (-1)^{i+j}(\sigma_{i,j}\otimes_Hid) f\left({[u_i * u_j]}_T, u_1, \dots, \widehat{u_i}, \dots, \widehat{u_j}, \dots, u_{n+1}\right),
\end{align*} where $\sigma_{1,i}(h_1\otimes\cdots\otimes  h_n)=h_i\otimes h_1\otimes\cdots\otimes\hat{h_i}\otimes\cdots\otimes h_n$, and $\sigma_{i,j}(h_1\otimes\cdots\otimes h_n)=h_i\otimes h_j\otimes h_1\otimes\cdots\otimes\hat {h_i}\otimes\cdots\otimes \hat{h_j}\otimes\cdots\otimes h_n$ and the pseudobracket ${[\cdot * \cdot]}_T$ on $M$ is defined by  $ {[u * v]}_T = [T u * v] + [u * T v] + \lambda ([u * v]).$ The corresponding cohomology groups are denoted by $H^*_{\mathrm{CE}}(M, L)$. 

Moreover, one can verify that the twisted differential $d_T$ coincides with the Chevalley--Eilenberg differential $\delta_{\mathrm{CE}}$. Hence the cohomology $H^\bullet_T(M, L)$ is isomorphic to the Chevalley--Eilenberg cohomology $H^\bullet_{\mathrm{CE}}(M, L)$. 

\section{Non-Abelian Extensions and Cohomology}\label{sec:extensions}
\label{sec:non_abelian_extensions}
In this section, we develop the theory of non-abelian extensions for Rota--Baxter Lie $H$-pseudoalgebras. After defining non-abelian extensions, we introduce the associated $2$-cocycle data. We show that, there is an equivalence relation between non-abelian extensions $\mathrm{Ext}_{\text{nab}}(L_T , M_S )$ and cohomology set of non-abelian $2$-cocycle $H^2_{\text{nab}}(L_T, M_S)$. We begin by recalling the notion of a non-abelian extension.
\begin{definition}
\label{def:non_ab_ext}
Let $(L_T , {[ \cdot * \cdot ]}_L)$ and $(M_S, [ \cdot * \cdot ]_M)$ be two Rota--Baxter Lie $H$-pseudoalgebras. A  \emph{non-abelian extension} of $L_T$ by $M_S$ is a Rota--Baxter Lie $H$-pseudoalgebra $(R_U, [ \cdot * \cdot ]_R)$ equipped with a short exact sequence in the category of Rota--Baxter Lie $H$-pseudoalgebras, i.e., 
\begin{align}\label{nonab}0 \longrightarrow M_S  \xrightarrow{i} R_U \xrightarrow{p} L_T   \longrightarrow 0.
\end{align}
We often denote such an extension simply by $R_U$. The maps $i$ and $p$ are inclusion and projection maps respectively.
\end{definition}

\begin{definition} 
\label{def:equiv_ext}
Two non-abelian extensions $R_U$ and $R'_{U'}$ of $L_T  $ by $M_S $ are said to be \emph{equivalent} if there exists an isomorphism $\Theta: R_U \to R'_{U'}$ of Rota--Baxter Lie $H$-pseudoalgebras such that makes the following diagram commutes:
\[
\begin{tikzcd}
    0 \arrow[r] & M_S  \arrow[r, "i"] \arrow[d, equal] & R_U \arrow[r, "p"] \arrow[d, "\Theta"] & L_T   \arrow[r] \arrow[d, equal] & 0 \\
    0 \arrow[r] & M_S  \arrow[r, "i'"] & R'_{U'} \arrow[r, "p'"] & L_T   \arrow[r] & 0
\end{tikzcd}
\]
We denote by $\mathrm{Ext}_{\text{nab}}(L_T , M_S )$ the set of all equivalence classes of non-abelian extensions of $L_T  $ by $M_S $.
\end{definition}
To classify these extensions, we fix a section $s: L \to R$, which is an $H$-linear section map satisfying $p \circ s = \mathrm{id}_L$. The choice of section allows us to decompose any element of $R$ as $u + s(x)$ for $u \in M$ and $x \in L$. This decomposition leads to the following key definition. 
\begin{definition}
\label{def:cocycle_data}
Given a non-abelian extension $R_U$ and a section $s$, we define a triple $(\chi, \psi,\Phi)$ of three fundamental maps. In which $\chi: L \otimes L \to H^{\otimes 2} \otimes_H M$ is  a skew-symmetric $H^{\otimes 2}$-linear map, $\psi: L \otimes M \to H^{\otimes 2} \otimes_H M$ is an $H^{\otimes 2}$-linear map defining a representation of $L$ on $M$,  and  $\Phi: L \to M$ is an $H$-linear map satisfying
\begin{equation}
\label{eq:cocycle_def}
\begin{cases}
    \chi(x, y) &:= {[s(x) * s(y)]}_R - (\mathrm{id}_{H^{\otimes 2}} \otimes_H s){[x * y]}_L, \\
    \psi_x(u) &:= {{[s(x) * u]}}_R, \\
    \Phi(x) &:= U(s(x)) -  s  (T(x)),
\end{cases}
\end{equation}for all $x, y \in L$ and $u \in M$. 
The map $\chi$ measures the failure of $s$ to be a Lie pesudoalgebra homomorphism, $\psi$ describes the action of $L$ on $M$ via the adjoint action in $R$, and $\Phi$ measures the failure of $s$ to preserve the Rota--Baxter operator.
\end{definition}
Viewing the exact sequence \eqref{nonab} as a non-abelian extension of the underlying Lie 
$H$-pseudoalgebra $L$ by $M$ (by forgetting the Rota--Baxter operators), it follows from \cite{Das-Nonabelian-cohomology} that the associated non-abelian $2$-cocycle of $L$ with values in $M$ satisfy following equations: 
\begin{align}\label{2cocycle1}\psi_x(\psi_y(u)) - ((12) \otimes_H \mathrm{id}_M)\psi_y(\psi_x(u)) = \psi_{{[x*y]}_L}(u) + {[\chi(x,y)*u]}_M\\
\label{2cocycle2} \psi_x(\chi(y,z)) -((12)\otimes_Hid) \psi_y(\chi(x,z))+ ((123)\otimes_Hid)\psi_z(\chi(x,y)) \nonumber\\
\qquad\qquad = \chi(x,{[y*z]}_L)-((12)\otimes_Hid)\chi(y,{[x*z]}_L)+((123)\otimes_Hid)\chi(z,{[x*y]}_L)  
\end{align}for all $x,y,z \in L$ and $ u \in M$. In terms of $L_T$ and $M_S$, the above expression can be expressed in the following form.
\begin{lemma} The maps $\chi$, $\psi$, and $\Phi$ defined above satisfy the following compatible conditions: for all $x,y \in L$ and $u \in M$,
\begin{align}\label{E1}
&\psi_{T(x)}(S(u)) = (\mathrm{id}_{H^{\otimes 2}} \otimes_H S)\Big( \psi_{T(x)}(u) + \psi_{x}(S(u)) + \lambda \psi_{x}(u) \Big) \\
&\qquad\qquad\qquad + (\mathrm{id}_{H^{\otimes 2}} \otimes_H S)\big({[\Phi(x) * u]}_M\big) - {[\Phi(x) * S(u)]}_M, \nonumber
\end{align}

\begin{align}\label{E2}
&\chi(T(x), T(y)) - (\mathrm{id}_{H^{\otimes 2}} \otimes_H S)\Big(\chi(T(x), y) + \chi(x, T(y)) + \lambda \chi(x, y)\Big) \\
&\quad - (\mathrm{id}_{H^{\otimes 2}} \otimes_H \Phi)\Big({[T(x) * y]}_L + {[x * T(y)]}_L + \lambda {[x * y]}_L\Big) \nonumber \\
&\quad + \psi_{T(x)}(\Phi(y)) - ((12)\otimes_H \mathrm{id}_M)\,\psi_{T(y)}(\Phi(x)) \nonumber \\
&\quad + (\mathrm{id}_{H^{\otimes 2}} \otimes_H S)\Big( ((12) \otimes_H \mathrm{id}_M)\,\psi_y(\Phi(x)) - \psi_x(\Phi(y)) + (\mathrm{id}_{H^{\otimes 2}} \otimes_H \Phi)({[x * y]}_L) \Big) \nonumber \\
&\quad + {[\Phi(x) * \Phi(y)]}_M = 0. \nonumber
\end{align}
\end{lemma} 
\begin{proof}Consider that
\begin{align*}
 &\psi_{T(x)} S(u) - (\mathrm{id}_{H^{\otimes 2}} \otimes_H S)\Big( \psi_{T(x)} u + \psi_{x} S(u) + \lambda \psi_{x} u\Big) - (\mathrm{id}_{H^{\otimes 2}} \otimes_H S)({[\Phi(x) * u]}_M) + {[\Phi(x) * S(u)]}_M \\
&={[s(T(x)) * S(u)]}_R - (\mathrm{id}_{H^{\otimes 2}} \otimes_H S)\Big({[s(T(x)) * u]}_R + {[s(x) * S(u)]}_R + \lambda {[s(x) * u]}_R\Big) \\
&\quad - (\mathrm{id}_{H^{\otimes 2}} \otimes_H S)({[\Phi(x) * u]}_M) + {[\Phi(x) * S(u)]}_M \\
&= {[s(T(x)) * S(u)]}_R - (\mathrm{id}_{H^{\otimes 2}} \otimes_H S)\Big({[s(T(x)) * u]}_R + {[s(x) * S(u)]}_R + \lambda {[s(x) * u]}_R\Big) \\
&\quad - (\mathrm{id}_{H^{\otimes 2}} \otimes_H S)({[U(s(x)) * u]}_M - {[s(T(x)) * u]}_M) + {[U(s(x)) * S(u)]}_M - {[s(T(x)) * S(u)]}_M \\
&= - (\mathrm{id}_{H^{\otimes 2}} \otimes_H S)\Big({[U(s(x)) * u]}_M + {[s(x) * S(u)]}_M + \lambda {[s(x) * u]}_M\Big) + {[U(s(x)) * S(u)]}_M \\
&= - (\mathrm{id}_{H^{\otimes 2}} \otimes_H S)\Big({[U(s(x)) * u]}_M + {[s(x) * S(u)]}_M + \lambda {[s(x) * u]}_M\Big) \\
&\quad + (\mathrm{id}_{H^{\otimes 2}} \otimes_H S)\Big({[U(s(x)) * u]}_M + {[s(x) * S(u)]}_M + \lambda {[s(x) * u]}_M\Big) \quad (\text{as } S = {U|}_{M})\\
&= 0.
\end{align*}
This proves the identity in Eq. \eqref{E1}.
To prove the identity in Eq.~\eqref{E2}, we observe that
\begin{align*}
&\chi(T(x), T(y)) - (\mathrm{id}_{H^{\otimes 2}} \otimes_H S)\Big(\chi(T(x), y) + \chi(x, T(y)) + \lambda \chi(x, y)\Big) \\
&\quad - (\mathrm{id}_{H^{\otimes 2}} \otimes_H \Phi)\Big({[T(x) * y]}_L + {[x * T(y)]}_L + \lambda {[x * y]}_L\Big) \\
&\quad + \psi_{T(x)}(\Phi(y)) - ((12) \otimes_H \mathrm{id}_M)\,\psi_{T(y)}(\Phi(x)) \\
&\quad + (\mathrm{id}_{H^{\otimes 2}} \otimes_H S)\Big( ((12) \otimes_H \mathrm{id}_M)\,\psi_y(\Phi(x)) - \psi_x(\Phi(y)) + (\mathrm{id}_{H^{\otimes 2}} \otimes_H \Phi)({[x * y]}_L) \Big) \\
&\quad + {[\Phi(x) * \Phi(y)]}_M \\
&= {[s(T(x)) * s(T(y))]}_R - (\mathrm{id}_{H^{\otimes 2}} \otimes_H s)({[T(x) * T(y)]}_L) \\
&\quad - (\mathrm{id}_{H^{\otimes 2}} \otimes_H S)\Big( {[s(T(x)) * s(y)]}_R - (\mathrm{id}_{H^{\otimes 2}} \otimes_H s)({[T(x) * y]}_L) \\
&\qquad\qquad\qquad\quad +{ [s(x) * s(T(y))]}_R - (\mathrm{id}_{H^{\otimes 2}} \otimes_H s)({[x * T(y)]}_L) \\
&\qquad\qquad\qquad\quad + \lambda {[s(x) * s(y)]}_R - \lambda (\mathrm{id}_{H^{\otimes 2}} \otimes_H s)({[x * y]}_L) \Big) \\
&\quad - (\mathrm{id}_{H^{\otimes 2}} \otimes_H U \circ s)\Big({[T(x) * y]}_L + {[x * T(y)]}_L + \lambda {[x * y]}_L\Big) \\
&\quad + (\mathrm{id}_{H^{\otimes 2}} \otimes_H s \circ T)\Big({[T(x) * y]}_L + {[x * T(y)]}_L + \lambda {[x * y]}_L\Big) \\
&\quad + {[s(T(x)) * U(s(y))]}_R - {[s(T(x)) * s(T(y))]}_R \\
&\quad - ((12)\otimes_H \mathrm{id}_M){[s(T(y)) * U(s(x))]}_R + ((12)\otimes_H \mathrm{id}_M){[s(T(y)) * s(T(x))]}_R \\
&\quad + (\mathrm{id}_{H^{\otimes 2}} \otimes_H S)((12) \otimes_H \mathrm{id}_M){[s(y) * U(s(x))]}_R \\
&\quad - (\mathrm{id}_{H^{\otimes 2}} \otimes_H S)((12)\otimes_H \mathrm{id}_M){[s(y) * s(T(x))]}_R \\
&\quad - (\mathrm{id}_{H^{\otimes 2}} \otimes_H S){[s(x) * U(s(y))]}_R + (\mathrm{id}_{H^{\otimes 2}} \otimes_H S){[s(x) * s(T(y))]}_R \\
&\quad + (\mathrm{id}_{H^{\otimes 2}} \otimes_H S \circ U \circ s)({[x * y]}_L) - (\mathrm{id}_{H^{\otimes 2}} \otimes_H S \circ s \circ T)({[x * y]}_L) \\
&\quad + {[U(s(x)) * U(s(y))]}_R - {[U(s(x)) * s(T(y))]}_R \\
&\quad - {[s(T(x)) * U(s(y))]}_R + {[s(T(x)) * s(T(y))]}_R.
\end{align*}
The above expression vanishes because both $T$ and $U$ satisfy the Rota--Baxter identity on the Lie $H$-pseudoalgebras $(L, {[\cdot * \cdot]}_L)$ and $(R, {[\cdot * \cdot]}_R)$, respectively. In our computation, we use that $M$ is an ideal in $R$, so ${[r * u]}_R \in H^{\otimes 2} \otimes_H M$ for all $r \in R$, $u \in M$, and that $S = U|_M$. This completes the proof.
\end{proof}
\begin{definition}
\label{def:na_cocycle}
A \emph{non-abelian $2$-cocycle} of $ L_T$ with values in $M_S$ is a triple $(\chi, \psi, \Phi)$ consisting of maps as defined above, satisfying the conditions \eqref{2cocycle1},\eqref{2cocycle2},\eqref{E1}, and \eqref{E2}.
\end{definition} 

For a split sequence in the category of $H$-modules, the section is not unique in general. However, a non-abelian $2$-cocycle depends on the section. For
two different sections, we have the following results.\\
Let $s': L \to R$ be any other section of the exact sequence in Eq. \eqref{nonab}. Define the map $\tau :L \to M$ by
\begin{align*}
 \tau(x) := s(x) - s'(x) \quad \text{for all } x \in L.
\end{align*}
Let $\chi'_\l, \rho', \Phi'$ be the maps deformed by the section $s'$, given by Eqs. \eqref{eq:cocycle_def}. For all $x,y \in L $ and $u \in M$, we have:\begin{align}\label{equivalent1}
\psi _x(u) - \psi' _x(u) &={ [s(x) *u]}_R - {[s'(x) *u]}_R = {[\tau(x) *u]}_M,
\end{align}
\begin{align}\label{equivalent2}
&\chi(x, y) - \chi'(x, y) \\&= {[s(x) * s(y)]}_R - (\mathrm{id}_{H^{\otimes 2}} \otimes_H s)({[x * y]}_L) - {[s'(x) * s'(y)]}_R + (\mathrm{id}_{H^{\otimes 2}} \otimes_H s')({[x * y]}_L) \nonumber \\
&= {[(s' + \tau)(x) * (s' + \tau)(y)]}_R - (\mathrm{id}_{H^{\otimes 2}} \otimes_H (s' + \tau))({[x * y]}_L)  - {[s'(x) * s'(y)]}_R + (\mathrm{id}_{H^{\otimes 2}} \otimes_H s')({[x * y]}_L) \nonumber \\
&= {[s'(x) * s'(y)]}_R + {[s'(x) * \tau(y)]}_R + {[\tau(x) * s'(y)]}_R + {[\tau(x) * \tau(y)]}_R  - (\mathrm{id}_{H^{\otimes 2}} \otimes_H s')({[x * y]}_L)\nonumber \\
&\quad - (\mathrm{id}_{H^{\otimes 2}} \otimes_H \tau)({[x * y]}_L)  - {[s'(x) * s'(y)]}_R + (\mathrm{id}_{H^{\otimes 2}} \otimes_H s')({[x * y]}_L) \nonumber \\
&= {[\tau(x) * \tau(y)]}_M - (\mathrm{id}_{H^{\otimes 2}} \otimes_H \tau)({[x * y]}_L) + {[s'(x) * \tau(y)]}_R + {[\tau(x) * s'(y)]}_R \nonumber \\
&= {[\tau(x) * \tau(y)]}_M - (\mathrm{id}_{H^{\otimes 2}} \otimes_H \tau)({[x * y]}_L) + \psi'_x(\tau(y)) - ((12)\otimes_H \mathrm{id}_M)\psi'_y(\tau(x)),\nonumber
\end{align}
and
\begin{align}\label{equivalent3} 
\Phi(x) - \Phi'(x) &= (U \circ s -  s \circ T)(x) - (U \circ s' -  s'  \circ T)(x)  \\
&= U(s - s')(x) -  (s - s') (T(x)) \nonumber \\
&= S(\tau(x)) - \tau(T(x)). \nonumber
\end{align} 
The above discussion leads us to the following definition.
\begin{definition}
    Two non-abelian $2$-cocycles $(\chi, \psi, \Phi)$ and $(\chi', \psi', \Phi')$ are said to be \emph{equivalent} if there exists an $H$-linear map $\tau:L \to M$ such that for all $x, y \in L$ and $u \in M$, we have
\begin{align}\psi_x(u) - \psi'_x(u) &= {[\tau(x)*u]}_M, \label{eq:cocycle_equiv_1} \\\chi(x, y) - \chi'(x, y) &= \psi'_x(\tau(y)) - ((12) \otimes_H \mathrm{id}_M)\psi'_y(\tau(x)) - (\mathrm{id}_{H^{\otimes 2}} \otimes_H \tau){[x*y]}_L + {[\tau(x)*\tau(y)]}_M, \label{eq:cocycle_equiv_2} \\\Phi(x) - \Phi'(x) &= S(\tau(x)) -  \tau (T(x)). \label{eq:cocycle_equiv_3} \end{align}
\end{definition}
We denote by $H^2_{\text{nab}}( L_T, M_S)$ the set of all equivalence classes of non-abelian 2-cocycles of $ L_T$ with values in $M_S$. We call it the \emph{non-abelian cohomology group}.
We now provide the main theorem of this section, which relates the set of all equivalence classes of non-abelian cohomology groups $H^2_{\text{nab}}( L_T, M_S)$ with the set of all equivalence classes of non-abelian extensions $\mathrm{Ext}_{\text{nab}}(L_T, M_S)$.   
\begin{theorem} 
\label{thm:classification}
There is a bijection between the set of equivalence classes of non-abelian extensions $\mathrm{Ext}_{\text{nab}}(L_T, M_S)$ and the non-abelian cohomology set $H^2_{\text{nab}}(L_T, M_S)$ in the category of Lie $H$-pseudoalgebras.
\end{theorem}
\begin{proof}
Let $R_U$ and $R'_{U'}$ be two equivalent non-abelian extensions of $L_T$ by $M_S$. If $s: L \to R$ is a section of the projection map, then $s' := \phi \circ s$ is a section of the projection map $proj'$ for $R'$. Let $(\chi', \psi', \Phi')$ be the non-abelian $2$-cocycle corresponding to the extension $R'_{U'}$ and section $s'$, then we have 
\begin{align*}
\chi'(x, y) &= {[s'(x) * s'(y)]}_{R'} - (\mathrm{id}_{H^{\otimes 2}} \otimes_H s')({[x * y]}_L) \\
&= {[(\phi \circ s)(x) * (\phi \circ s)(y)]}_{R'} - (\mathrm{id}_{H^{\otimes 2}} \otimes_H (\phi \circ s))({[x * y]}_L) \\
&= (\mathrm{id}_{H^{\otimes 2}} \otimes_H \phi)\left({[s(x) * s(y)]}_R - (\mathrm{id}_{H^{\otimes 2}} \otimes_H \phi)(\mathrm{id}_{H^{\otimes 2}} \otimes_H s)({[x * y]}_L)\right)\\& = (\mathrm{id}_{H^{\otimes 2}} \otimes_H \phi)\chi(x, y)\\&=\chi(x, y) \quad \text{(as $\phi|_M = \mathrm{id}_M$)},
\end{align*}
\begin{align*}
\psi' _x(u) &= {[s'(x) *u]}_{R'} = {[(\phi \circ s)(x) *u]}_{R'}= {[(\phi \circ s)(x) * \phi(u)]}_{R'} \quad \text{(as $\phi(u)=u$)}\\&
= (\mathrm{id}_{H^{\otimes 2}} \otimes_H \phi)({[s(x) *u]}_R) = (\mathrm{id}_{H^{\otimes 2}} \otimes_H \phi)\psi _x(u)=\psi _x(u)\quad \text{(as $\phi|_M = \mathrm{id}_M$)},
\end{align*}
\begin{align*}
\Phi'(x) &= U(s'(x)) - s'(T(x)) \\
&= U(\phi(s(x))) - \phi(s(T(x))) \\
&= \phi(U(s(x))) - \phi(s(T(x))) \quad \text{(since $\phi$ is a morphism of Rota--Baxter pseudoalgebras)} \\
&= \phi\big(U(s(x)) - s(T(x))\big) \\
&= \phi(\Phi(x)) = \Phi(x) \quad \text{(as $\Phi(x) \in M$ and $\phi|_M = \mathrm{id}_M$).}
\end{align*}

Thus, equivalent extensions give the same cohomology class in $H^2_{\text{nab}}(L_T, M_S)$. In other words, we say that there is a well-defined map $\Upsilon : \mathrm{Ext}_{\mathrm{nab}}(L_T, M_S) \to H^2_{\mathrm{nab}}(L_T, M_S)$ which assigns an equivalence class of extensions to the equivalence class of the corresponding non-abelian $2$-cocycles.

Conversely, given a non-abelian $2$-cocycle $(\chi, \psi, \Phi)$, define the $H$-module $ J := L \oplus M$ equipped with the pseudobracket:
\begin{align*}
{[(x,u) * (y, v)]}_ J &:= \left({[x * y]}_L, \psi _x(v) - ((12)\otimes_H \mathrm{id}_M)\psi_y(u) + \chi(x, y) + {[u * v]}_M\right)
\end{align*}
for $(x,u), (y, v) \in  J$. The cocycle conditions ensure that $(J,{[.*.]}_J)$ is a    Lie $H$-pseudoalgebra.

We further define an $ H$-linear map $ O :  J \to  J$ by
\begin{align*}
 O(x,u) := (T(x), S(u) + \Phi(x)),\quad \text{ for all } (x,u) \in  J.
\end{align*} Then 
\begin{small}
    \begin{align*} &{ [O(x,u) * O(y, v)]}_ J \\
 &= {\left[\left( T(x), S(u) + \Phi(x) \right) * \left( T(y), S(v) + \Phi(y) \right)\right]}_ J \\
 &= \Bigg( {[T(x) * T(y)]}_L, \\
 &\quad \overbrace{\psi_{T(x)}(S(v))}^{T1} + \overbrace{\psi_{T(x)}(\Phi(y))}^{T2} - \overbrace{((12) \otimes_H \mathrm{id}_M)\psi_{T(y)}(S(u))}^{T3} - \overbrace{((12) \otimes_H \mathrm{id}_M)\psi_{T(y)}(\Phi(x))}^{T4} \\
 &\quad + \overbrace{\chi(T(x), T(y))}^{T5} + \overbrace{{[S(u) * S(v)]}_M}^{T6} + \overbrace{{[S(u) * \Phi(y)]}_M}^{T7} + \overbrace{{[\Phi(x) * S(v)]}_M}^{T8} + \overbrace{{[\Phi(x) * \Phi(y)]}_M}^{T9} \Bigg)\\
 &= \Bigg((\mathrm{id}_{H^{\otimes 2}} \otimes_H T)({[T(x)*y]}_L+{[x*T(y)]}_L+\lambda({[x*y]}_L)), \\
 &\quad \overbrace{(\mathrm{id}_{H^{\otimes 2}} \otimes_H S)(\psi_{T(x)}(v) + \psi _x(S(v)) + \lambda \psi _x(v)) + (\mathrm{id}_{H^{\otimes 2}} \otimes_H S)({[\Phi(x) * v]}_M)}^{T1+T8} \\
 &\quad - \overbrace{(\mathrm{id}_{H^{\otimes 2}} \otimes_H S)(((12)\otimes_H \mathrm{id}_M)\psi_{T(y)}(u) + ((12) \otimes_H \mathrm{id}_M)\psi_y(S(u)) + \lambda ((12) \otimes_H \mathrm{id}_M)\psi_y(u)) - (\mathrm{id}_{H^{\otimes 2}} \otimes_H S)({[\Phi(y) *u]}_M)}^{T3+T7} \\
 &\quad + \overbrace{(\mathrm{id}_{H^{\otimes 2}} \otimes_H S)({[S(u) * v]}_M +{ [u * S(v)]}_M + \lambda {[u * v]}_M)}^{T6} \\
 &\quad + \overbrace{(\mathrm{id}_{H^{\otimes 2}} \otimes_H S)(\chi(T(x), y) + \chi(x, T(y)) + \lambda \chi(x, y)) + (\mathrm{id}_{H^{\otimes 2}} \otimes_H \Phi)({[T(x) * y]}_L + {[x * T(y)]}_L + \lambda {[x * y]}_L)}^{T2+T4+T5+T9} \\
 &\quad \overbrace{ +(\mathrm{id}_{H^{\otimes 2}} \otimes_H S)(\psi _x(\Phi(y)) - ((12)\otimes_H \mathrm{id}_M)\psi_y(\Phi(x)) + (\mathrm{id}_{H^{\otimes 2}} \otimes_H \Phi)({[x * y]}_L)) \Bigg) }^{T2+T4+T5+T9} \\
 &= \Bigg( (\mathrm{ id}_{H^{\otimes 2}} \otimes_H T)({[T(x)* y]}_L), \\
&\quad (\mathrm{ id}_{H^{\otimes 2}} \otimes_H S)\left(\psi_{T(x)}(v) - ((12) \otimes_H \mathrm{id}_M)\psi_y(S(u) + \Phi(x)) + \chi(T(x), y) + {[(S(u) + \Phi(x)) * v]}_M\right)  + (\mathrm{id}_{H^{\otimes 2}} \otimes_H \Phi)({[T(x) * y]}_L) \Bigg) \\
&\quad + \Bigg((\mathrm{ id}_{H^{\otimes 2}} \otimes_H T)({[x * T(y)]}_L), \\
&\quad (\mathrm{ id}_{H^{\otimes 2}} \otimes_H S)\left(\psi _x(S(v) + \Phi(y)) - ((12)\otimes_H \mathrm{id}_M)\psi_{T(y)}(u) + \chi(x, T(y)) + {[u * (S(v) + \Phi(y))]}_M\right)  + (\mathrm{id}_{H^{\otimes 2}} \otimes_H \Phi)({[x * T(y)]}_L) \Bigg) \\
&\quad + \lambda\Bigg((\mathrm{id}_{H^{\otimes 2}} \otimes_H T)({[x * y]}_L), \\
&\quad (\mathrm{id}_{H^{\otimes 2}} \otimes_H S)\left(\psi _x(v) - ((12)\otimes_H \mathrm{id}_M)\psi_y(u) + \chi(x, y) + {[u * v]}_M\right) + (\mathrm{id}_{H^{\otimes 2}} \otimes_H \Phi)({[x * y]}_L) \Bigg) \\
&= (\mathrm{ id}_{H^{\otimes 2}} \otimes_H O)\left({[T(x) * y]}_L, \psi_{T(x)}(v) - ((12)\otimes_H \mathrm{id}_M)\psi_y(S(u) + \Phi(x)) + \chi(T(x), y) + {[(S(u) + \Phi(x)) * v]}_M\right) \\
&\quad + (\mathrm{id}_{H^{\otimes 2}} \otimes_H O)\left({[x * T(y)]}_L, \psi _x(S(v) + \Phi(y)) - ((12)\otimes_H \mathrm{id}_M)\psi_{T(y)}(u) + \chi(x, T(y)) + {[u * (S(v) + \Phi(y))]}_M\right)\\& \quad+ \lambda (\mathrm{id}_{H^{\otimes 2}} \otimes_H O)\left({[x * y]}_L, \psi _x(v) - ((12) \otimes_H \mathrm{id}_M)\psi_y(u) + \chi(x, y) + {[u * v]}_M\right) \\
&= (\mathrm{id}_{H^{\otimes 2}} \otimes_H O)\left({[O(x,u) * (y, v)]}_J + {[(x,u) * O(y, v)]}_J + \lambda {[(x,u) * (y, v)]}_J\right).
\end{align*}
\end{small} Thus, we see that $O$ is a Rota--Baxter operator on $J = L \oplus M $ and the corresponding Rota--Baxter Lie $H$-pseudoalgebra $(J, {[\cdot * \cdot]}_J, O) $ is denoted by $ J_O $.
 
Moreover, one can easily verify that the exact sequence $0 \longrightarrow M_S \xrightarrow{\mathrm{inc}} J_O \xrightarrow{\mathrm{proj}} L_T \longrightarrow 0$ is a non-abelian extension of the Rota--Baxter Lie $H$-pseudoalgebra $L_T$ by $M_S$, where $\mathrm{inc}(u) = (0,u)$ and $\mathrm{proj}(x,u) = x$ for all $(x,u) \in  J$ and $u \in M$.

Next, let $(\chi, \psi, \Phi)$ and $(\chi', \psi', \Phi')$ be two equivalent non-abelian $2$-cocycles, i.e., there exists an $H$-linear map $\tau : L \to M$,  such that the identities in Eqs. \eqref{equivalent1}, \eqref{equivalent2}, and \eqref{equivalent3} hold. Let $J'_{O'}$ be the Rota--Baxter Lie $H$-pseudoalgebra deformed by the $2$-cocycle $(\chi', \psi', \Phi')$. Note that the Lie $H$-pseudoalgebra $J' = L \oplus M$ is equipped with the pseudobracket given by
\[
{[(x,u) * (y, v)]}_{J'} := \left({[x * y]}_L, \psi' _x(v) - ((12)\otimes_H \mathrm{id}_M)\psi'_y(u) + \chi'(x, y) + {[u* v]}_M\right),
\]
for $(x,u), (y, v) \in  J'$. Moreover, the map $O' : J' \to J'$ is given by
\[
O'(x,u) := (T(x), S(u) + \Phi'(x)) \quad \text{for all } (x,u) \in J'.
\]
We now define a map $\phi : L \oplus M \to L \oplus M$ by
\[
\phi(x,u) := (x,u + \tau(x)) \quad \text{for all } (x,u) \in L \oplus M.
\]
Then, by a straightforward calculation, we show that
\begin{align*}
&(\mathrm{ id}_{H^{\otimes 2}} \otimes_H \phi )({[(x,u) * (y, v)]}_J) \\
&=(\mathrm{ id}_{H^{\otimes 2}} \otimes_H \phi )\left({[x * y]}_L, \psi _x(v) - ((12)\otimes_H \mathrm{id}_M)\psi_y(u) + \chi(x, y) + {[u* v]}_M\right) \\
&= \left({[x * y]}_L, (\mathrm{id}_{H^{\otimes 2}} \otimes_H \tau)({[x * y]}_L) + \psi _x(v) - ((12)\otimes_H \mathrm{id}_M)\psi_y(u) + \chi(x, y) + {[u * v]}_M\right) \\
&= \left({[x * y]}_L, (\mathrm{id}_{H^{\otimes 2}} \otimes_H \tau)({[x * y]}_L) + \psi' _x(v) + {[\tau(x) * v]}_M - ((12)\otimes_H \mathrm{id}_M)\psi'_y(u) - ((12)\otimes_H \mathrm{id}_M){[\tau(y) *u]}_M \right. \\
&\quad \left. + \chi'(x, y) + \psi' _x(\tau(y)) - ((12)\otimes_H \mathrm{id}_M)\psi'_y(\tau(x)) - (\mathrm{id}_{H^{\otimes 2}} \otimes_H \tau)({[x * y]}_L) + {[\tau(x) * \tau(y)]}_M + {[u * v]}_M\right) \\
&= \left({[x * y]}_L, \psi' _x(v) + \psi' _x(\tau(y)) - ((12)\otimes_H \mathrm{id}_M)\psi'_y(\tau(x)) - ((12)\otimes_H \mathrm{id}_M)\psi'_y(u) \right. \\
&\quad \left. + {[\tau(x) * v]}_M - ((12)\otimes_H \mathrm{id}_M){[\tau(y) *u]}_M + \chi'(x, y) + {[\tau(x) * \tau(y)]}_M + {[u * v]}_M\right) \\
&={ [(x,u + \tau(x)) * (y, v + \tau(y))]}_{J'} \\
&= {[\phi(x,u) * \phi(y, v)]}_{J'}.
\end{align*}
Additionally,
\begin{align*}
(O' \circ \phi)(x,u) &= O'(x,u + \tau(x)) \\
&= (T(x), S(u) + S(\tau(x)) + \Phi'(x)) \\
&= (T(x), S(u) + \tau(T(x)) + \Phi(x)) \quad (\text{by Eq. \eqref{equivalent3}}) \\
&= \phi(T(x), S(u) + \Phi(x)) = (\phi \circ O)(x,u).
\end{align*}
Hence, the map $\phi$ defines an equivalence between the two non-abelian extensions. Therefore, we obtain a well-defined map $\Lambda : H^2_{\mathrm{nab}}(L_T, M_S) \to \mathrm{Ext}_{\mathrm{nab}}(L_T, M_S)$. Finally, it is straightforward to verify that the maps $\Upsilon$ and $\Lambda$ are inverses of each other. This completes the proof.
\end{proof}
\section{Automorphisms of Rota--Baxter Lie $H$-pseudoalgebras and their inducibility}\label{sec:automorphisms}
In this section, we study the inducibility of a pair of automorphisms in the context of a non-abelian extension of Rota--Baxter Lie $H$-pseudoalgebras and present its relation with the Wells map. 

Let $L_T = (L, {[\cdot * \cdot]}_L, T)$ and $M_S = (M, {[\cdot * \cdot]}_M, S)$ be two Rota--Baxter Lie $H$-pseudoalgebras and 
\begin{align*}
 0 \longrightarrow M_S \xrightarrow{\mathrm{inc}} R_U \xrightarrow{\mathrm{proj}} L_T \longrightarrow 0
\end{align*} 
be a non-abelian extension of $L_T$ by $M_S$ with the section $s$ of $\mathrm{proj}$-map. Let $\mathrm{Aut}_M(R_U)$ be the set of all Rota--Baxter Lie $H$-pseudoalgebra automorphisms $\gamma \in \mathrm{Aut}(R_U)$ that satisfies $\gamma|_M \subset M$, i.e. $M$ is an invariant $H$-module. More explicitly,
\[
\mathrm{Aut}_M(R_U) := \{\gamma \in \mathrm{Aut}(R_U) \mid \gamma(M) = M\}.
\]

For any section $s : L \to R$ of the $\mathrm{proj}$-map, for any $\gamma \in \mathrm{Aut}_M(R_U)$, we can define an $H$-module map $\bar{\gamma}: L \to L$ by $\bar{\gamma}(x) := \mathrm{proj} \circ \gamma \circ s(x)$, for $x \in L$. It is easy to verify that the map $\bar{\gamma}$ is independent of the choice of the section $s$. Moreover, $\mathrm{proj}$ is a projection on $L$ and $\gamma$ preserves $L$, so $\bar{\gamma}$ is a bijection on $L$. For any $x, y \in L$, we have
\begin{align*}
(\mathrm{ id}_{H^{\otimes 2}} \otimes_H \bar{\gamma} )({[x* y]}_L) &= (\mathrm{ id}_{H^{\otimes 2}} \otimes_H \mathrm{proj} \circ \gamma \circ s) ({[x* y]}_L) \\&= (\mathrm{ id}_{H^{\otimes 2}} \otimes_H \mathrm{proj} \circ \gamma) ((\mathrm{ id}_{H^{\otimes 2}} \otimes_H s){[x* y]}_L) \\
&= (\mathrm{ id}_{H^{\otimes 2}} \otimes_H \mathrm{proj} \circ \gamma)({[s(x) * s(y)]}_R - \chi(x, y)) \\
&= (\mathrm{ id}_{H^{\otimes 2}} \otimes_H \mathrm{proj} \circ \gamma)({[s(x) * s(y)]}_R) \quad (\because \gamma|_M \subset M \text{ and } \mathrm{proj}|_M = 0) \\
&= {[\mathrm{proj} \circ \gamma \circ s(x) * \mathrm{proj} \circ \gamma \circ s(y)]}_L = {[\bar{\gamma}(x) * \bar{\gamma}(y)]}_L,
\end{align*}
and 
\begin{align*}
(T \circ \bar{\gamma} - \bar{\gamma} \circ T)(x)
&= T(\mathrm{proj}(\gamma(s(x)))) - \mathrm{proj}(\gamma(s(T(x)))) \\
&= \mathrm{proj}(U(\gamma(s(x)))) - \mathrm{proj}(\gamma(U(s(x)))) \\
&= \mathrm{proj}\big( \gamma(U(s(x)) - s(T(x))) \big) \\
&= \mathrm{proj}(\gamma(\Phi(x))) = 0.
\end{align*}   This shows that the map $\bar{\gamma}: L_T \to L_T$ is an automorphism of the Rota--Baxter Lie $H$-pseudoalgebra $L_T$. In other words, $\bar{\gamma} \in \mathrm{Aut}(L_T)$. Observe that $\overline{\gamma_1 \gamma_2} = \overline{\gamma_1} \ \overline{\gamma_2}$.

Now we obtain a group homomorphism
\[
\Pi : \mathrm{Aut}_M(R_U) \to \mathrm{Aut}(M_S) \times \mathrm{Aut}(L_T) \quad \text{given by} \quad \gamma \mapsto (\gamma|_M, \bar{\gamma}).
\]
In general, we say any pair $(\alpha, \beta) \in \mathrm{Aut}(M_S) \times \mathrm{Aut}(L_T)$ is called inducible, if $(\alpha, \beta)$ lies in the image of $\Pi$ defined above.
\begin{proposition}\label{prop6.1}
Consider the non-abelian extension of Rota--Baxter Lie $H$-pseudoalgebras. For a section $s$, suppose the extension corresponds to the non-abelian $2$-cocycle $(\chi, \psi, \Phi)$. Then a pair $(\alpha, \beta) \in \mathrm{Aut}(M_S) \times \mathrm{Aut}(L_T)$ of Rota--Baxter Lie $H$-pseudoalgebra automorphisms is inducible if and only if there exists an $H$-linear map $\eta : L \to M$ satisfying the following conditions:
\begin{align}
(\mathrm{id}_{H^{\otimes 2}} \otimes_H \alpha)(\psi_x(u)) - \psi_{\beta(x)}(\alpha(u)) &= {[\eta(x) * \alpha(u)]}_M, \label{eq:16} \\
(\mathrm{id}_{H^{\otimes 2}} \otimes_H \alpha)(\chi(x, y)) - \chi(\beta(x), \beta(y)) &= \psi_{\beta(x)}(\eta(y)) - ((12)\otimes_H \mathrm{id}_M)\psi_{\beta(y)}(\eta(x))\nonumber \\
&\quad - (\mathrm{id}_{H^{\otimes 2}} \otimes_H \eta)({[x* y]}_L) + {[\eta(x) * \eta(y)]}_ M, \label{eq:17} \\
 \alpha(\Phi(x)) - \Phi(\beta(x)) &=  S(\eta(x)) - \eta(T(x)). \label{eq:18}
\end{align}
for all $x, y \in L$ and $u \in M$.
\end{proposition}

\begin{proof}
Let $(\alpha, \beta)$ be an inducible pair, i.e., there exists a Rota--Baxter Lie $H$-pseudoalgebra automorphism $\gamma \in \mathrm{Aut}_M(R_U)$ such that $\gamma|_M = \alpha$ and $\mathrm{proj} \circ \gamma \circ s = \beta$. For any $x \in L$, we observe that $(\gamma \circ s - s \circ \beta)(x) \in \ker(\mathrm{proj})$, which in turn implies that $(\gamma \circ s - s \circ \beta)(x) \in M$. We define a map $\eta : L \to M$ by 
\[
\eta(x) := (\gamma \circ s - s \circ \beta)(x), \quad \text{for } x \in L.
\]
We observe that, for $x \in L$ and $u \in M$ 
\begin{align*}
(\mathrm{id}_{H^{\otimes 2}} \otimes_H \alpha)(\psi_x(u)) - \psi_{\beta(x)}(\alpha(u)) 
&= (\mathrm{id}_{H^{\otimes 2}} \otimes_H \alpha)({[s(x) * u]}_R) - {[s(\beta(x)) * \alpha(u)]}_R \\
&= {[\gamma(s(x)) * \gamma(u)]}_R - {[s(\beta(x)) * \alpha(u)]}_R \quad (\text{as } \gamma|_M = \alpha) \\
&= {[\gamma(s(x)) * \alpha(u)]}_R - {[s(\beta(x)) * \alpha(u)]}_R \\
&= {[(\gamma(s(x)) - s(\beta(x))) * \alpha(u)]}_R \\
&= {[\eta(x) * \alpha(u)]}_M,
\end{align*}
  
\begin{align*}
&\psi_{\beta(x)}(\eta(y)) - ((12)\otimes_H \mathrm{id}_M)\psi_{\beta(y)}(\eta(x)) - (\mathrm{id}_{H^{\otimes 2}} \otimes_H \eta)({[x* y]}_L) + {[\eta(x) * \eta(y)]}_M \\
&= {[s(\beta(x)) * (\gamma \circ s - s \circ \beta)(y)]}_R - ((12)\otimes_H \mathrm{id}_M){[s(\beta(y)) * (\gamma \circ s - s \circ \beta)(x)]}_R \\
&\quad - (\mathrm{id}_{H^{\otimes 2}} \otimes_H (\gamma \circ s - s \circ \beta))({[x* y]}_L) + {[(\gamma \circ s - s \circ \beta)(x) * (\gamma \circ s - s \circ \beta)(y)]}_M \\
&= \cancelto{2}{{[s(\beta(x)) * \gamma(s(y))]}_R} -\cancelto{1}{ {[s(\beta(x)) * s(\beta(y))]}_R} \\
&\quad- \cancelto{3}{((12)\otimes_H \mathrm{id}_M){[s(\beta(y)) * \gamma(s(x))]}_R}  + ((12)\otimes_H \mathrm{id}_M){[s(\beta(y)) * s(\beta(x))]}_R\\&\quad -
(\mathrm{id}_{H^{\otimes 2}} \otimes_H \gamma)((\mathrm{id}_{H^{\otimes 2}} \otimes_H  s )({[x* y]}_L)) + ((\mathrm{id}_{H^{\otimes 2}} \otimes_H  s )(\mathrm{id}_{H^{\otimes 2}} \otimes_H \beta )({[x* y]}_L) \\
&\quad + {[\gamma(s(x)) * \gamma(s(y))]}_R - \cancelto{3}{{[\gamma(s(x)) * s(\beta(y))]}_R} - \cancelto{2}{{[s(\beta(x)) * \gamma(s(y))]}_R} + \cancelto{1}{{[s(\beta(x)) * s(\beta(y))]}_R} \\
&
=-(\mathrm{id}_{H^{\otimes 2}} \otimes_H id_M){[s(\beta(x)) * s(\beta(y))]}_R\\&\quad -
(\mathrm{id}_{H^{\otimes 2}} \otimes_H \gamma)((\mathrm{id}_{H^{\otimes 2}} \otimes_H  s )({[x* y]}_L)) + ((\mathrm{id}_{H^{\otimes 2}} \otimes_H  s )(\mathrm{id}_{H^{\otimes 2}} \otimes_H \beta )({[x* y]}_L) \\
&\quad + {[\gamma(s(x)) * \gamma(s(y))]}_R  \\
&= -{[s(\beta(x)) * s(\beta(y))]}_R+ ((\mathrm{id}_{H^{\otimes 2}} \otimes_H  s )(\mathrm{id}_{H^{\otimes 2}} \otimes_H \beta )({[x* y]}_L)\\&\quad -
(\mathrm{id}_{H^{\otimes 2}} \otimes_H \gamma)((\mathrm{id}_{H^{\otimes 2}} \otimes_H  s )({[x* y]}_L))   + {[\gamma(s(x)) * \gamma(s(y))]}_R  \\
&= - \chi(\beta(x), \beta(y))+ (\mathrm{id}_{H^{\otimes 2}} \otimes_H \alpha)(\chi(x, y)),
\end{align*}
and
\begin{align*}
&\alpha(\Phi(x)) - \Phi(\beta(x)) \\
&= \alpha\big(U(s(x)) - s(T(x))\big) - \big(U(s(\beta(x))) - s(T(\beta(x)))\big) \\
&= \gamma\big(U(s(x)) - s(T(x))\big) - U(s(\beta(x))) + s(T(\beta(x))) \quad \text{(since $\alpha = \gamma|_M$)} \\
&= \gamma(U(s(x))) - \gamma(s(T(x))) - U(s(\beta(x))) + s(T(\beta(x))) \\
&= U(\gamma(s(x))) - U(s(\beta(x))) - \gamma(s(T(x))) + s(T(\beta(x))) \quad \text{(as $\gamma$ commutes with $U$)} \\
&= U(\gamma(s(x)) - s(\beta(x))) - \big(\gamma(s(T(x))) - s(T(\beta(x)))\big) \\
&= U(\eta(x)) - \eta(T(x)) \quad \text{(by definition of $\eta$)} \\
&= S(\eta(x)) - \eta(T(x)) \quad \text{(since $U|_M = S$).}
\end{align*}
 
Conversely, suppose an $H$-linear map $\eta : L \to M$ satisfies Eqs. \eqref{eq:16}, \eqref{eq:17}, and \eqref{eq:18}. First, observe that, using the section $s$, the $H$-module $R$ can be identified with $s(L) \oplus M$. Equivalently, any element $f \in R$ can be uniquely written as $f = s(x) + u$, for some $x \in L$ and $u \in M$. Using this, we define a map $\gamma : R \to R$ by
\[
\gamma(f) = \gamma(s(x) + u) = s(\beta(x)) + (\alpha(u) + \eta(x)), \quad \text{for } f = s(x) + u \in R.
\]
Since $s$, $\alpha$, $\beta$ are all $H$-linear injections, we can easily verify that $\gamma$ is also an $H$-linear injective map. The maps $\alpha$ and $\beta$ are also surjective, which implies that $\gamma$ is also a surjective map. Thus, $\gamma$ becomes a bijective map. By using Eqs. \eqref{eq:16} and \eqref{eq:17}, it is easy to show that $\gamma : R \to R$ is a Lie $H$-pseudoalgebra homomorphism. 
 
For all $f = s(x) + u, g = s(y) + v \in R$, we have

\begin{align*}
 &(\mathrm{id}_{H^{\otimes 2}} \otimes_H \gamma){[f * g]}_R \\&= (\mathrm{id}_{H^{\otimes 2}} \otimes_H \gamma){[(s(x) + u )* (s(y) + v)]}_R \\
&=(\mathrm{id}_{H^{\otimes 2}} \otimes_H \gamma) \Big({[s(x) * s(y)]}_R + {[s(x) * v]}_R + {[u * s(y)]}_R + {[u * v]}_R \Big)\\
&= (\mathrm{id}_{H^{\otimes 2}} \otimes_H \gamma) \left((\mathrm{id}_{H^{\otimes 2}} \otimes_H s)({[x* y]}_L) + \chi(x,y)\right)\\& + (\mathrm{id}_{H^{\otimes 2}} \otimes_H \gamma) \psi_x(v) - (\mathrm{id}_{H^{\otimes 2}} \otimes_H \gamma) ((12)\otimes_H \mathrm{id}_M)\psi_y(u) + (\mathrm{id}_{H^{\otimes 2}} \otimes_H \gamma) {[u * v]}_M.
\\
&= (\mathrm{id}_{H^{\otimes 2}} \otimes_H s \circ \beta)({[x* y ]}_L) + (\mathrm{id}_{H^{\otimes 2}} \otimes_H \eta)({[x* y]}_L) + (\mathrm{id}_{H^{\otimes 2}} \otimes_H \alpha)(\chi(x,y)) \quad (\text{ Since } \gamma|_M = \alpha). \\& + (\mathrm{id}_{H^{\otimes 2}} \otimes_H \alpha)(\psi_x(v))  - ((12) \otimes_H \alpha)(\psi_y(u))+ {[\alpha(u) * \alpha(v)]}_M.
\end{align*} Now using Eq. \eqref{eq:16} we have $(\mathrm{id}_{H^{\otimes 2}} \otimes_H \alpha)(\psi_x(v)) = \psi_{\beta(x)}(\alpha(v)) + {[\eta(x) * \alpha(v)]}_M.$

And using Eq. \eqref{eq:17}, we have
\begin{align*}
(\mathrm{id}_{H^{\otimes 2}} \otimes_H \alpha)(\chi(x,y)) &= \chi(\beta(x), \beta(y)) + \psi_{\beta(x)}(\eta(y)) - ((12)\otimes_H \mathrm{id}_M)\psi_{\beta(y)}(\eta(x)) \\
&\quad - (\mathrm{id}_{H^{\otimes 2}} \otimes_H \eta)({[x* y]}_L) + {[\eta(x) * \eta(y)]}_M.
\end{align*}

After substituting the corresponding terms 
\begin{align*}
&= (\mathrm{id}_{H^{\otimes 2}} \otimes_H s \circ \beta)({[x* y]}_L) + \cancel{(\mathrm{id}_{H^{\otimes 2}} \otimes_H \eta)({[x* y]}_L)} \\
&\quad + \chi(\beta(x), \beta(y)) + \psi_{\beta(x)}(\eta(y)) - ((12)\otimes_H \mathrm{id}_M)\psi_{\beta(y)}(\eta(x)) - \cancel{(\mathrm{id}_{H^{\otimes 2}} \otimes_H \eta)({[x* y]}_L)} + {[\eta(x) * \eta(y)]}_M \\
&\quad + \psi_{\beta(x)}(\alpha(v)) + {[\eta(x) * \alpha(v)]}_M - ((12)\otimes_H \mathrm{id}_M)\psi_{\beta(y)}(\alpha(u)) - ((12)\otimes_H \mathrm{id}_M){[\eta(y) * \alpha(u)]}_M \\
&\quad + {[\alpha(u) * \alpha(v)]}_M \\
&= (\mathrm{id}_{H^{\otimes 2}} \otimes_H s)({[\beta(x) * \beta(y)]}_L) + \chi(\beta(x), \beta(y)) \\
&\quad + \psi_{\beta(x)}(\alpha(v) + \eta(y)) - ((12)\otimes_H \mathrm{id}_M)\psi_{\beta(y)}(\alpha(u) + \eta(x)) \\
&\quad + {[\alpha(u) + \eta(x) * \alpha(v) + \eta(y)]}_M \\
&= {[s(\beta(x)) * s(\beta(y))]}_R + {[s(\beta(x)) * (\alpha(v) + \eta(y))]}_R + {[(\alpha(u) + \eta(x)) * s(\beta(y))]}_R \\
&\quad + {[(\alpha(u) + \eta(x)) * (\alpha(v) + \eta(y))]}_M \\&= {[s(\beta(x)) + (\alpha(u) + \eta(x)) * s(\beta(y)) + (\alpha(v) + \eta(y))]}_R\\
&= {[\gamma(f) * \gamma(g)]}_R.
\end{align*}
Moreover, for any $f = s(x) + u \in R$, we observe that
\begin{align*}
\gamma(U(f)) 
&= \gamma(U(s(x) + u)) \\
&= \gamma(U(s(x)) + S(u)) \\
&= \gamma\big( s(T(x)) + \Phi(x) + S(u) \big) \quad (\text{as } \Phi = U \circ s - s \circ T) \\
&= \gamma(s(T(x))) + \gamma(\Phi(x)) + \gamma(S(u)) \\
&= s(\beta(T(x))) + \eta(T(x)) + \alpha(\Phi(x)) + \alpha(S(u)) \\
&= s(T(\beta(x))) + \eta(T(x)) + \alpha(\Phi(x)) + S(\alpha(u)) \quad (\text{since } T\beta = \beta T \text{ and } \alpha \in \mathrm{Aut}(M_S)) \\
&= \Big( s(T(\beta(x))) + \Phi(\beta(x)) \Big) + S(\alpha(u)) + \Big( \alpha(\Phi(x)) - \Phi(\beta(x)) + \eta(T(x)) \Big) \\
&= U(s(\beta(x))) + S(\alpha(u)) + S(\eta(x)) \quad (\text{by Eq.~}\eqref{eq:18}) \\
&= U(s(\beta(x))) + S(\alpha(u) + \eta(x)) \\
&= U\big( s(\beta(x)) + \alpha(u) + \eta(x) \big) \quad (\text{since } U|_M = S) \\
&= U(\gamma(f)).
\end{align*}
Finally, from the definition of $\gamma$, we have that $\gamma(M) \subset M$. Thus, $\gamma \in \mathrm{Aut}_M(R_U)$. Additionally, one can easily verify that $\gamma|_M = \alpha$ and $\mathrm{proj} \circ \gamma \circ s = \beta$. This shows that the pair $(\alpha, \beta)$ is inducible.
\end{proof}
Note that the necessary and sufficient condition derived in Proposition \ref{prop6.1} depends on the choice of the section map 
$s$. To obtain a criterion independent of this choice, we adapt a method originally developed by Wells \cite{Wells-Automorphisms} in the context of abstract group extensions.

Given any pair $(\alpha,\beta)\in \mathrm{Aut}(M_S) \times \mathrm{Aut}(L_T)$ of automorphisms, we define a new triple $(\chi^{(\alpha,\beta)}, \psi^{(\alpha,\beta)}, \Phi^{(\alpha,\beta)})$ of $H$-linear maps by
\begin{align}\label{newcocycle}
\chi^{(\alpha,\beta)}(x, y) &:= (\mathrm{id}_{H^{\otimes 2}} \otimes_H \alpha)\left(\chi(\beta^{-1}(x), \beta^{-1}(y))\right), \\
\psi^{(\alpha,\beta)}_x(u) &:= (\mathrm{id}_{H^{\otimes 2}} \otimes_H \alpha)\left(\psi_{\beta^{-1}(x)}(\alpha^{-1}(u))\right), \\
\Phi^{(\alpha,\beta)}(x) &:= \alpha \left(\Phi(\beta^{-1}(x))\right),
\end{align}
for $x, y \in L$ and $u \in M$.
\begin{lemma}
The triple $(\chi^{(\alpha,\beta)}, \psi^{(\alpha,\beta)}, \Phi^{(\alpha,\beta)})$ is a non-abelian $2$-cocycle on $L_T$ with values in $M_S$.
\end{lemma}
\begin{proof}
This follows by direct verification using the cocycle conditions \eqref{2cocycle1}
to \eqref{E2}.
\end{proof}
\begin{definition}
The \emph{Wells map} $\mathfrak{W} : \mathrm{Aut}(M_S) \times \mathrm{Aut}(L_T) \to H^2_{\text{nab}}(L_T, M_S)$ is defined by \begin{align*}
    \mathfrak{W}(\alpha,\beta) = \left[(\chi^{(\alpha,\beta)}, \psi^{(\alpha,\beta)}, \Phi^{(\alpha,\beta)}) - (\chi, \psi, \Phi)\right],
\end{align*}
where the bracket denotes the equivalence class in non-abelian cohomology.
\end{definition}

Note that any non-abelian $2$-cocycle depends on the choice of section, but the Wells map does not. This follows from the fact that equivalent $2$-cocycles belong to the same cohomology class. Since the Wells map is defined in terms of these cohomology classes rather than specific cocycles, it is independent of the choice of section.

\begin{proposition}
The Wells map $\mathfrak{W}$ is independent of the choice of section.
\end{proposition}
\begin{proof}
Let $s'$ be any other section of the projection $\mathrm{proj}$, and let $(\chi', \psi', \Phi')$ be the corresponding non-abelian $2$-cocycle. The cocycles $(\chi, \psi, \Phi)$ and $(\chi', \psi', \Phi')$ are equivalent via the map $\tau := s - s' : L \to M$, satisfying:
\begin{align*}
\chi'(x,y) - \chi(x,y) &= \psi_x(\tau(y)) - ((12)\otimes_H \mathrm{id}_M)\psi_y(\tau(x)) - (\mathrm{id}_{H^{\otimes 2}} \otimes_H \tau)({[x*y]}_L) + {[\tau(x) * \tau(y)]}_M, \\
\psi'_x(u) - \psi_x(u) &= {[\tau(x) * u]}_M, \\
\Phi'(x) - \Phi(x) &= S(\tau(x)) -  \tau(T(x)).
\end{align*}
Now consider the twisted cocycles. Define $\tau^{(\alpha,\beta)} := \alpha \circ \tau \circ \beta^{-1} : L \to M$. We claim that $(\chi^{(\alpha, \beta)}, \psi^{(\alpha, \beta)}, \Phi^{(\alpha, \beta)})$ and $(\chi'^{(\alpha, \beta)}, \psi'^{(\alpha, \beta)}, \Phi'^{(\alpha, \beta)})$ are equivalent via $\tau^{(\alpha,\beta)}$.


For example, starting from the $\Phi$-equivalence:
\begin{align*}\Phi'^{\alpha,\beta}(\beta^{-1}(x)) - \Phi^{\alpha,\beta}(\beta^{-1}(x)) &=\alpha\left(\Phi'(\beta^{-1}(x))\right)- \alpha\left(\Phi(\beta^{-1}(x))\right)\\& 
=\alpha\left(\Phi'(\beta^{-1}(x)) - \Phi(\beta^{-1}(x))\right)\\&=
\alpha
\left(S(\tau(\beta^{-1}(x))) -  \tau(T(\beta^{-1}(x)))\right)\\&= \alpha\left(S(\tau(\beta^{-1}(x)))\right) - \alpha\left( \tau(T(\beta^{-1}(x)))\right)\quad \text{(since $\alpha S = S\alpha$)} \\
&= S(\alpha(\tau(\beta^{-1}(x)))) -  \alpha \circ \tau (T(\beta^{-1}(x))) \\
&= S(\tau^{(\alpha,\beta)}(x)) -   \alpha \circ \tau (\beta^{-1}(T(x))) \quad \text{(since $T\beta = \beta T$)} \\
&= S(\tau^{(\alpha,\beta)}(x)) - \tau^{(\alpha,\beta)} (T(x)).
\end{align*}The other equivalence conditions follow similarly. The cocycles $(\chi^{(\alpha,\beta)}, \psi^{(\alpha,\beta)}, \Phi^{(\alpha,\beta)})$ and $(\chi, \psi, \Phi)$ are equivalent, and similarly for the primed versions. Therefore, they define the same cohomology class in $H^2_{\mathrm{nab}}(L_T, M_S)$.
\end{proof}
\begin{theorem}\label{thm6.4}
Let $L_T$ and $M_S$ be two Rota--Baxter Lie $H$-pseudoalgebras and 
\begin{align*}
0 \longrightarrow M_S \xrightarrow{\mathrm{inc}} R_U \xrightarrow{\mathrm{proj}} L_T \longrightarrow 0
\end{align*} 
be a non-abelian extension of $L_T$ by $M_S$ with a section $s$. The corresponding non-abelian $2$-cocycle of $R_U$ is denoted by $(\chi, \psi, \Phi)$. A pair $(\alpha, \beta) \in \mathrm{Aut}(M_S) \times \mathrm{Aut}(L_T)$ is inducible if and only if the non-abelian $2$-cocycles $(\chi, \psi, \Phi)$ and $(\chi^{(\alpha,\beta)}, \psi^{(\alpha,\beta)}, \Phi^{(\alpha,\beta)})$ are equivalent. In other words, the Wells map of the pair $(\alpha,\beta)$ is zero.
\end{theorem}
\begin{proof}
Let $(\alpha, \beta) \in \mathrm{Aut}(M_S) \times \mathrm{Aut}(L_T)$ be an inducible pair of Rota--Baxter Lie $H$-pseudoalgebra automorphisms. For any fixed section $s$, let the given non-abelian extension produce the non-abelian $2$-cocycle $(\chi, \psi, \Phi)$. Then by Proposition \ref{prop6.1}, there exists an $H$-linear map $\eta : L \to M$ satisfying Eqs. \eqref{eq:16}-\eqref{eq:18}. In these identities, if we replace $x,y,u$ respectively by $\beta^{-1}(x), \beta^{-1}(y), \alpha^{-1}(u)$, we get that
\begin{align*}
\psi^{(\alpha,\beta)}_x(u) - \psi_x(u) &= {[\eta(\beta^{-1}(x)) * u]}_M, \\
\chi^{(\alpha,\beta)}(x, y) - \chi(x, y) &= \psi_x(\eta(\beta^{-1}(y))) - ((12)\otimes_H \mathrm{id}_M)\psi_y(\eta(\beta^{-1}(x))) \\
&\quad - (\mathrm{id}_{H^{\otimes 2}} \otimes_H \eta \circ \beta^{-1})({[x* y]}_L) + {[\eta(\beta^{-1}(x)) * \eta(\beta^{-1}(y))]}_M, \\
\Phi^{(\alpha,\beta)}(x) - \Phi(x) &= S(\eta(\beta^{-1}(x))) -  \eta \circ \beta^{-1} (T(x)).
\end{align*}
This shows that the non-abelian $2$-cocycles $(\chi^{(\alpha,\beta)}, \psi^{(\alpha,\beta)}, \Phi^{(\alpha,\beta)})$ and $(\chi, \psi, \Phi)$ are equivalent by the map $\eta \circ \beta^{-1}$. Hence, they represent the same class in $H^2_{\mathrm{nab}}(L_T, M_S)$, i.e., $\mathfrak{W}(\alpha,\beta) = 0$.

Conversely, assume that $\mathfrak{W}(\alpha, \beta) = 0$. As before, let $s$ be any section and $(\chi, \psi, \Phi)$ be the non-abelian $2$-cocycle produced from the given non-abelian extension. Then it follows that the non-abelian $2$-cocycles $(\chi^{(\alpha,\beta)}, \psi^{(\alpha,\beta)}, \Phi^{(\alpha,\beta)})$ and $(\chi, \psi, \Phi)$ are equivalent
. One can easily verify that the map $\eta:= \tau \circ \beta : L \to M$ satisfies Eqs. \eqref{eq:16}-\eqref{eq:18}. Therefore, by Proposition \ref{prop6.1}, the pair $(\alpha, \beta)$ is inducible.
\end{proof}
We now generalize the exact sequence of Wells maps, for the Rota--Baxter Lie $H$-pseudoalgebras. Given any non-abelian extension of Rota--Baxter Lie $H$-pseudoalgebras $0 \longrightarrow M_S \xrightarrow{\mathrm{inc}} R_U \xrightarrow{\mathrm{proj}} L_T \longrightarrow 0$, define a subgroup $\mathrm{Aut}^{L,M}_M(R_U) \subset \mathrm{Aut}_M(R_U)$ by $\mathrm{Aut}^{L,M}_M(R_U):= \{\gamma \in \mathrm{Aut}_M(R_U) \mid \Pi(\gamma) = (\mathrm{id}_M, \mathrm{id}_L)\}.$

Thus, we have the following theorem.

\begin{theorem}\label{thm:wells-rb}
Let 
\[
0 \longrightarrow M_S \xrightarrow{inc} R_U\xrightarrow{proj} L_T \longrightarrow 0
\]
be a non-abelian extension of Rota--Baxter Lie $H$-pseudoalgebras. Then there is an exact sequence
\[
1 \longrightarrow \mathrm{Aut}^{L,M}_{M}(R_U) \xrightarrow{\iota} \mathrm{Aut}_{M}(R_U) \xrightarrow{\Pi} \mathrm{Aut}(M_S) \times \mathrm{Aut}(L_T) \xrightarrow{\mathfrak{W}} H^2_{nab}(L_T, M_S).
\]
\end{theorem} 
 \section{ Homotopy Rota--Baxter Lie $H$-pseudoalgebras} \label{sec:homotopy}

The homotopy theory of pseudoalgebras is a relatively recent development; see, for instance, \cite{Das-2025} for the case of higher Lie Lie $H$-pseudoalgebras. In this section, we introduce the concept of a $2$-term  $L_\infty$-$H$-pseudoalgebra equipped with a homotopy Rota--Baxter operator and provide some important results.
 
\begin{definition}
An \textbf{ $L_\infty$-$H$-pseudoalgebra} (or strongly homotopy Lie $H$-pseudoalgebra) is a pair $(L, \{\beta_k\}_{k\geq1})$ consisting of a graded left $H$-module $L = \bigoplus_{n \in \mathbb{Z}} L_n$, equipped with a collection of polylinear maps $\{\beta_k \in \mathrm{Hom}_{H^{\otimes k}}(L^{\boxtimes k}, H^{\otimes k} \otimes_H L)\}_{k\geq 1}$ of degree $k-2$, satisfying the following conditions:

\begin{enumerate}
    \item \textbf{Graded Skew-symmetry}: Each map $\beta_k$ is graded skew-symmetric. For any homogeneous elements $x_1, \dots, x_k \in L$ and any permutation $\sigma \in S_k$, we have
    \[
        \beta_k(x_1, \dots, x_k) = (-1)^{\sigma}\epsilon(\sigma)  (\sigma \otimes_H \mathrm{id}_L) \beta_k(x_{\sigma^{-1}(1)}, \dots, x_{\sigma^{-1}(k)}),
    \]
    where $\epsilon(\sigma)$ is the Koszul sign associated to the permutation of the homogeneous elements $x_i$.
    
    \item \textbf{Higher Jacobi Identities}: For every $N \in \mathbb{N}$ and for all homogeneous elements $x_1, \dots, x_N \in L$, the following identity holds:
    \[
        \sum_{\substack{i+j=N+1 \\ i,j \geq 1}} \sum_{\sigma \in \mathrm{Sh}(i-1,j)} (-1)^{(j-1)(i-1)} \epsilon(\sigma) (\sigma\otimes_Hid)\beta_j(\beta_i(x_{\sigma(1)}, \dots, x_{\sigma(i)}), x_{\sigma(i+1)}, \dots, x_{\sigma(N)}) = 0.
    \]
    Here, $\mathrm{Sh}(i-1,j)$ denotes the set of $(i-1,j)$-shuffles, and the action of $\beta_j$ on the result of $\beta_i$ is defined using the pseudotensor category structure of $M^*(H)$.
\end{enumerate}
\end{definition}
  
\begin{definition}
A \textbf{$2$-term  $L_\infty$-$H$-pseudoalgebra} is a triple consisting of:
\begin{enumerate}
    \item[(i)] A chain complex of left $H$-modules $\beta_1: L_1 \to L_0$,
    \item[(ii)] A graded skew-symmetric, $H^{\otimes 2}$-linear map $\beta_2 \in \mathrm{Hom}_{H^{\otimes 2}}(L_i \boxtimes L_j, H^{\otimes 2} \otimes_H L_{i+j})$, for $i,j,i+j\in [0, 1]$,
    \item[(iii)] A graded skew-symmetric, $H^{\otimes 3}$-linear map $\beta_3 \in \mathrm{Hom}_{H^{\otimes 3}}(L_0 \boxtimes L_0 \boxtimes L_0, H^{\otimes 3} \otimes_H L_{1})$,
\end{enumerate}
that satisfy the following set of identities:
\begin{enumerate}
    \item[(L1)] $\beta_2(u , v) = 0$ for all $u, v \in L_1$,
    \item[(L2)] $\beta_2(x , u) = -((12) \otimes_H \mathrm{id}_{L_1})\beta_2(u , x)$ for all $x \in L_0$, $u \in L_1$,
    \item[(L3)] $\beta_2(x , y) = -((12)\otimes_H \mathrm{id}_{L_0})\beta_2(y , x)$ for all $x, y \in L_0$,
    \item[(L4)] $\beta_1(\beta_2(x , u)) = \beta_2(x , \beta_1(u))$ for all $x \in L_0$, $u \in L_1$,
    \item[(L5)] $\beta_2(\beta_1(u) , v) = \beta_2(u ,\beta_1(v))$ for all $u, v \in L_1$,
    \item[(L6)] $-\beta_1(\beta_3(x, y, z)) = \beta_2(\beta_2(x , y), z) -\beta_2(x , \beta_2(y , z)) + ((12)\otimes_Hid)  \beta_2(y , \beta_2(x ,z))$ for all $x, y, z \in L_0$,
    \item[(L7)] $-\beta_3(x, y, \beta_1(u)) = \beta_2(\beta_2(x , y) , u) -\beta_2(x , \beta_2(y , u)) +((12)\otimes_Hid)\beta_2(y , \beta_2(x , u))$ for all $x, y \in L_0$, $u \in L_1$,
    \item[(L8)] The Jacobiator identity holds: 
    \begin{align*}
        &\beta_2(x , \beta_3(y, z, w)) - ((12)\otimes_Hid)\beta_2(y , \beta_3(x, z, w)) + ((123)\otimes_Hid)\beta_2(z , \beta_3(x, y, w)) \\& -((1234)\otimes_Hid)\beta_2(w , \beta_3(x, y, z))
        - \beta_3(\beta_2(x , y) , z , w) +((23)\otimes_Hid)\beta_3(\beta_2(x , z) , y , w) \\&
       - ((234)\otimes_Hid)\beta_3(\beta_2(x , w) , y , z) - ((132)\otimes_Hid)\beta_3(\beta_2(y , z) , x , w)+((1342)\otimes_Hid)\beta_3(\beta_2(y , w) , x , z) \\&-((13)(24)\otimes_Hid)\beta_3(\beta_2(z , w) , x , y) =0,
    \end{align*}
    for all $x,y,z,w\in L_0$ and $u,v\in L_1$. Here, the action of $\beta_2$ on the result of $\beta_3$ is defined using the pseudotensor category structure.
\end{enumerate}
\end{definition}
\begin{definition}
Let $\mathcal{L} = (L_1 \xrightarrow{\beta_1} L_0, \beta_2, \beta_3)$ and $\mathcal{L}' = (L'_1 \xrightarrow{\beta'_1} L'_0, \beta'_2, \beta'_3)$ be two $2$-term $L_\infty$-$H$-pseudoalgebras. A \emph{morphism} between them is given by a triple $(f_0, f_1, f_2)$, where: $f_0: L_0 \to L'_0$ and $f_1: L_1 \to L'_1$ are $H$-linear maps, and  $f_2 \in \mathrm{Hom}_{H^{\otimes 2}}(L_0 \boxtimes L_0, H^{\otimes 2} \otimes_H L'_1)$ is an $H^{\otimes 2}$-linear  graded skew-symmetric map,
such that the following equations hold for all homogeneous elements $x, y \in L_0$ and $u \in L_1$:
\begin{enumerate}
    \item[(H1)] $f_{0} \circ \beta_1 = \beta'_1 \circ f_1,$ 
    \item[(H2)] $ (id_{H^{\otimes 2}}\otimes_H \beta'_1)(f_2(x, y)) = -\beta'_2(f_0(x) , f_0(y)) + (id_{H^{\otimes 2}} \otimes_H f_0)(\beta_2(x ,y))$,
    \item[(H3)] $f_2(x, \beta_1(u)) =  (id_{H^{\otimes 2}} \otimes_H f_1)(\beta_2(x ,u)) - \beta'_2(f_0(x) ,f_1(u)),$ 
    \item[(H4)] $f_2(\beta_1(u), x) = - (id_{H^{\otimes 2}} \otimes_H f_1)(\beta_2(u ,x)) + \beta'_2(f_1(u) ,f_0(x)),$
    \item[(H5)] \begin{align*}
        &-(id_{H^{\otimes 3}} \otimes_H f_1)(\beta_3(x, y, z)) + \beta'_3(f_0(x), f_0(y), f_0(z)) = \\
        &\qquad -\beta'_2(f_0(x) ,f_2(y, z)) - f_2(x, \beta_2(y ,z)) +((12)\otimes_Hid)\beta'_2(f_0(y) ,f_2(x, z)) \\
        &\qquad + ((12)\otimes_Hid)f_2(y, \beta_2(x ,z)) + \beta'_2(f_2(x, y) ,f_0(z)) + f_2(\beta_2(x ,y) ,z).
    \end{align*}
\end{enumerate}
We often denote such a morphism by $(f_0, f_1, f_2): \mathcal{L} \to \mathcal{L}'$. The composition of two morphisms $(f_0, f_1, f_2): \mathcal{L} \to \mathcal{L}'$ and $(g_0, g_1, g_2): \mathcal{L}' \to \mathcal{L}''$ is given by $(g_0 \circ f_0, g_1 \circ f_1, (g_2 \circ (f_0 \otimes f_0)) + (id_{H^{\otimes 2}} \otimes_H g_1) \circ f_2)$.
\end{definition}Further suppose that $\mathcal{L} = (L_1 \xrightarrow{\beta_1} L_0, \beta_2, \beta_3)$ is a $2$-term  $L_\infty$-$H$-pseudoalgebra. Then the identity morphism $\mathrm{Id}_\mathcal{L}: \mathcal{L} \to \mathcal{L}$ is given by the triple 
\[
\mathrm{Id}_\mathcal{L} = (\mathrm{id}_{L_0}, \mathrm{id}_{L_1}, 0).
\]
This corresponds to the following identities for all $x, y, z \in L_0$ and $u \in L_1$:
\begin{itemize}
    \item[a)] $\mathrm{id}_{L_0} \circ \beta_1 = \beta_1 \circ \mathrm{id}_{L_1},$
    \item[b)] $(\mathrm{id}_{H^{\otimes 2}} \otimes_H \mathrm{id}_{L_0})(\beta_2(x, y)) - \beta_2(\mathrm{id}_{L_0}(x), \mathrm{id}_{L_0}(y)) = 0,$
    \item[c)] $(\mathrm{id}_{H^{\otimes 2}} \otimes_H \mathrm{id}_{L_1})(\beta_2(x, u)) - \beta_2(\mathrm{id}_{L_0}(x), \mathrm{id}_{L_1}(u)) = 0,$
    \item[d)] $(\mathrm{id}_{H^{\otimes 2}} \otimes_H \mathrm{id}_{L_1})(\beta_2(u, x)) - \beta_2(\mathrm{id}_{L_1}(u), \mathrm{id}_{L_0}(x)) = 0,$
    \item[e)] $(\mathrm{id}_{H^{\otimes 3}} \otimes_H \mathrm{id}_{L_1})(\beta_3(x, y, z)) - \beta_3(\mathrm{id}_{L_0}(x), \mathrm{id}_{L_0}(y), \mathrm{id}_{L_0}(z)) = 0.$
\end{itemize}
These identities are trivially satisfied because all the maps involved are identity maps.

\begin{definition}
Let $(L_1 \xrightarrow{\beta_1} L_0, \beta_2, \beta_3)$ be a $2$-term  $L_\infty$-$H$-pseudoalgebra. A \emph{homotopy Rota--Baxter operator of weight $\lambda \in \mathbb{k}$} on this structure is a triple $\mathcal{T} = (\mathcal{T}_0, \mathcal{T}_1, \mathcal{T}_2)$ consisting of two $H$-linear maps $\mathcal{T}_0 : L_0 \to L_0$, $\mathcal{T}_1 : L_1 \to L_1$, and a graded skew-symmetric $H^{\otimes 2}$-linear map 
\[
\mathcal{T}_2 \in \mathrm{Hom}_{H^{\otimes 2}}(L_0 \boxtimes L_0,\, H^{\otimes 2} \otimes_H L_1),
\]
subject to the following identities for all $x, y, z \in L_0$ and $u \in L_1$:
\begin{align}
    &\beta_1 \circ \mathcal{T}_1 = \mathcal{T}_0 \circ \beta_1, \label{eq:rb_identity1} \\
&\beta_1(\mathcal{T}_2(x, y)) = (\mathrm{id}_{H^{\otimes 2}} \otimes_H \mathcal{T}_0)\Big( \beta_2(\mathcal{T}_0(x), y) + \beta_2(x, \mathcal{T}_0(y)) + \lambda \beta_2(x, y) \Big) - \beta_2(\mathcal{T}_0(x), \mathcal{T}_0(y)), \label{eq:rb_identity2} \\
&\mathcal{T}_2(x, \beta_1(u)) = (\mathrm{id}_{H^{\otimes 2}} \otimes_H \mathcal{T}_1)\Big( \beta_2(\mathcal{T}_0(x), u) + \beta_2(x, \mathcal{T}_1(u)) + \lambda \beta_2(x, u) \Big) - \beta_2(\mathcal{T}_0(x), \mathcal{T}_1(u)), \label{eq:rb_identity3} \\
&\beta_2\Big((\mathcal{T}_0(x), \mathcal{T}_2(y, z)) -  (\mathcal{T}_0(y), \mathcal{T}_2(x, z)) -  (\mathcal{T}_2(x, y), \mathcal{T}_0(z))\Big) \nonumber \\
&\quad + \mathcal{T}_2\big( \beta_2(\mathcal{T}_0(x), y) + \beta_2(x, \mathcal{T}_0(y)) + \lambda \beta_2(x, y),\, z \big) \nonumber \\
&\quad + \mathcal{T}_2\big( x,\, \beta_2(\mathcal{T}_0(y), z) + \beta_2(y, \mathcal{T}_0(z)) + \lambda \beta_2(y, z) \big) \nonumber \\
&\quad - \mathcal{T}_2\big( y,\, \beta_2(\mathcal{T}_0(x), z) + \beta_2(x, \mathcal{T}_0(z)) + \lambda \beta_2(x, z) \big) \nonumber \\
&\quad + (\mathrm{id}_{H^{\otimes 3}} \otimes_H \mathcal{T}_1)\Big( 
\beta_2(x, \mathcal{T}_2(y, z)) - ((12)\otimes_Hid)\beta_2(y, \mathcal{T}_2(x, z)) - \beta_2(\mathcal{T}_2(x, y), z) \nonumber \\
&\qquad - \mathcal{T}_2(\beta_2(x, y), z) + \mathcal{T}_2(x, \beta_2(y, z)) -((12)\otimes_Hid) \mathcal{T}_2(y, \beta_2(x, z)) 
\Big) \nonumber \\
&= \beta_3(\mathcal{T}_0(x), \mathcal{T}_0(y), \mathcal{T}_0(z)) - (\mathrm{id}_{H^{\otimes 3}} \otimes_H \mathcal{T}_1)\Big( 
\beta_3(\mathcal{T}_0(x), \mathcal{T}_0(y), z) + \beta_3(\mathcal{T}_0(x), y, \mathcal{T}_0(z)) + \beta_3(x, \mathcal{T}_0(y), \mathcal{T}_0(z)) \nonumber \\
&\qquad + \lambda \big( \beta_3(\mathcal{T}_0(x), y, z) + \beta_3(x, \mathcal{T}_0(y), z) + \beta_3(x, y, \mathcal{T}_0(z)) \big) 
+ \lambda^2 \beta_3(x, y, z)
\Big). \label{eq:rb_identity4}
\end{align}
Here, all compositions are defined using the pseudotensor category structure of $M^*(H)$.
\end{definition}
 
A $2$-term  $L_\infty$-$H$-pseudoalgebra equipped with a homotopy Rota--Baxter operator $\mathcal{T} =(\mathcal{T}_0, \mathcal{T}_1, \mathcal{T}_2)$ is referred to as a \emph{$2$-term Rota--Baxter  $L_\infty$-$H$-pseudoalgebra}. We denote such a structure by $\mathcal{L}_\mathcal{T} = (\beta_1: L_1\to L_0, \beta_2, \beta_3, \mathcal{T}_0, \mathcal{T}_1, \mathcal{T}_2)$ or simply by $\mathcal{L}_\mathcal{T}$.

\subsection{Skeletal and Strict $2$-Term Rota--Baxter $L_\infty$-$H$-pseudoalgebras}

\begin{definition}
Let $\mathcal{L}_\mathcal{T}$ be a $2$-term Rota--Baxter  $L_\infty$-$H$-pseudoalgebra. It is said to be \emph{skeletal} if $\beta_1 =0$ and is said to be \emph{strict} if $\beta_3 =0$ and $\mathcal{T}_2 =0$.
\end{definition}

Note that, if $\mathcal{L}_\mathcal{T}$ is \emph{skeletal}, then we have
\begin{enumerate}
    \item[(sk1)] $\beta_2(u ,v) = 0$,
    \item[(sk2)] $\beta_2(x ,u) = -((12)\otimes_H \mathrm{id}_{L_1})\beta_2(u ,x)$,
    \item[(sk3)] $\beta_2(x ,y) = -((12) \otimes_H \mathrm{id}_{L_0})\beta_2(y ,x)$,
    \item [(sk4)] $0= -\beta_2(x ,\beta_2(y ,z)) + \beta_2(\beta_2(x ,y) ,z) + ((12)\otimes_Hid)\beta_2(y ,\beta_2(x ,z))$,
    \item [(sk5)] $0= -\beta_2(x ,\beta_2(y ,u)) + \beta_2(\beta_2(x ,y) ,u) +((12)\otimes_Hid)\beta_2(y ,\beta_2(x ,u))$,
    \item[(sk6)] The higher Jacobi identity (L8) holds for $\beta_3$,
    \item [(sk7)] 
$ 0=(\mathrm{id}_{H^{\otimes 2}} \otimes_H \mathcal{T}_1)\Big( \beta_2(\mathcal{T}_0(x) ,u) + \beta_2(x ,\mathcal{T}_1(u))+ \lambda \beta_2( x  , u) \Big) -  \beta_2(\mathcal{T}_0(x) ,\mathcal{T}_1(u)), $
    \item[(sk8)] 
 $0 = (\mathrm{id}_{H^{\otimes 2}} \otimes_H \mathcal{T}_0)\Big( \beta_2(\mathcal{T}_0(x) ,y) + \beta_2(x ,\mathcal{T}_0(y)) + \lambda\beta_2 (x  ,y) \Big)-
    \beta_2(\mathcal{T}_0(x) ,\mathcal{T}_0(y)), $
  \item[(sk9)] Equation~\eqref{eq:rb_identity4} holds with $\beta_1 = 0$,
\end{enumerate}
for all homogeneous elements $x,y,z,w\in L_0$ and $u,v\in L_1$. 

\par Note that $\mathcal{L}_\mathcal{T}$ is \emph{strict} if $\beta_3 =0$ and $\mathcal{T}_2=0$, then we have  
\begin{enumerate}
    \item[(st1)] $\beta_2(u ,v) = 0$,
    \item[(st2)] $\beta_2(x ,u) = -((12)\otimes_H \mathrm{id}_{L_1})\beta_2(u ,x)$,
    \item[(st3)] $\beta_2(x ,y) = -((12) \otimes_H \mathrm{id}_{L_0})\beta_2(y ,x)$,
    \item [(st4)]$\beta_1( \beta_2(x ,u)) = \beta_2(x ,\beta_1(u))$,
    \item [(st5)]$ \beta_2(\beta_1(u) ,v) = \beta_2(u ,\beta_1(v))$,
    \item [(st6)] $0=- \beta_2(x ,\beta_2(y ,z)) + \beta_2(\beta_2(x ,y) ,z) + ((12)\otimes_Hid)\beta_2(y ,\beta_2(x ,z))$,
    \item [(st7)]$0= -\beta_2(x ,\beta_2(y ,u)) + \beta_2(\beta_2(x ,y) ,u) +((12)\otimes_Hid)\beta_2(y ,\beta_2(x ,u))$,
    \item [(st8)]$\beta_1 \circ \mathcal{T}_1 = \mathcal{T}_0 \circ \beta_1,$
    \item [(st9)] $0=(\mathrm{id}_{H^{\otimes 2}} \otimes_H \mathcal{T}_1)\Big( \beta_2(\mathcal{T}_0(x) ,u) + \beta_2(x ,\mathcal{T}_1(u)) + \lambda\beta_2( x  , u) \Big) -  \beta_2(\mathcal{T}_0(x) ,\mathcal{T}_1(u)),$
    \item [(st10)]$0= (\mathrm{id}_{H^{\otimes 2}} \otimes_H \mathcal{T}_0)\Big( \beta_2(\mathcal{T}_0(x) ,y) + \beta_2(x ,\mathcal{T}_0(y)) + \lambda\beta_2 (x  ,y) \Big)-
    \beta_2(\mathcal{T}_0(x) ,\mathcal{T}_0(y)), $
\end{enumerate}
for all homogeneous elements $x,y,z,w\in L_0$ and $u,v\in L_1$. 

Next, we introduce the crossed module of Rota--Baxter Lie $H$-pseudoalgebras and characterize strict $2$-term $L_\infty$-$H$-pseudoalgebras.
\begin{definition}
A crossed module of Rota--Baxter Lie $H$-pseudoalgebras is a quadruple $( {L_0}_{\mathcal{T}_0},{L_1}_{\mathcal{T}_1}, t, \rho)$, where ${L_0}_{\mathcal{T}_0} = (L_0, {[\cdot*\cdot]}_{L_0}, \mathcal{T}_0)$ and ${L_1}_{\mathcal{T}_1} = (L_1, {[\cdot*\cdot]}_{L_1}, \mathcal{T}_1)$ are both Rota--Baxter Lie $H$-pseudoalgebras, $t: {L_1}\to {L_0}$ is a morphism of Rota--Baxter Lie $H$-pseudoalgebras, and $\rho:{L_0} \boxtimes {L_1} \to H^{\otimes 2} \otimes_H {L_1}$ is an $H^{\otimes 2}$-linear map such that $(L_1, \rho)$ is a representation of $(L_0, {[\cdot*\cdot]}_{L_0})$, satisfying:\begin{align*} 
    (\mathrm{id}_{H^{\otimes 2}} \otimes_H  t)(\rho(x ,u)) &= { [x* t(u)]}_{L_0} , \\ 
    \rho(t(u) ,v) &=  {[u * v ]}_{L_1},
\end{align*}
for all $x\in L_0$ and $u, v \in L_1$.
\end{definition}
Let $((L_0, {[\cdot * \cdot]}_{L_0}, \mathcal{T}_0), (L_1, {[\cdot * \cdot]}_{L_1}, \mathcal{T}_1), t, \rho)$ be a crossed module of Rota--Baxter Lie $H$-pseudoalgebras. Define the deformed brackets:
\[
{[x* y]}_{\mathcal{T}_0} := {[\mathcal{T}_0(x)* y]}_{L_0} + {[x* \mathcal{T}_0(y)]}_{L_0} + \lambda {[x* y]}_{L_0}, \quad
{[u* v]}_{\mathcal{T}_1} := {[\mathcal{T}_1(u)* v]}_{L_1} + {[u* \mathcal{T}_1(v)]}_{L_1} + \lambda {[u*v]}_{L_1}.
\] Then for any $ u, v \in L_1 $, we observe that
\begin{align*}
(\mathrm{id}_{H^{\otimes 2}} \otimes_H  t)\big([u * v]^{\mathcal{T}_1}_{L_1}\big) &=  (\mathrm{id}_{H^{\otimes 2}} \otimes_H  t)\Big( {[\mathcal{T}_1(u) * v]}_{L_1} + {[u * \mathcal{T}_1(v)]}_{L_1} + \lambda {[u*v]}_{L_1}\Big) \\
&=  \Big( {[t(\mathcal{T}_1(u)) * t(v)]}_{L_0} + {[t(u) * t(\mathcal{T}_1(v))]}_{L_0} + \lambda{[t(u) * t(v)]}_{L_0} \Big) \\
&=  \Big( {[\mathcal{T}_0(t(u)) * t(v)]}_{L_0} + {[t(u) * \mathcal{T}_0(t(v))]}_{L_0} + \lambda{[t(u) * t(v)]}_{L_0} \Big) \\
&= [t(u) * t(v)]^{\mathcal{T}_0}_{L_0} \quad (\because t\circ \mathcal{T}_1 = \mathcal{T}_0\circ t),
\end{align*}
which shows that $ t : L_1^{\mathcal{T}_1} \to L_0^{\mathcal{T}_0} $ is a homomorphism of deformed Lie $H$-pseudoalgebras. 
Next, we define the map $ \rho_1 : L_0^{\mathcal{T}_0} \boxtimes L_1^{\mathcal{T}_1}\to H^{\otimes 2} \otimes_H L_1^{\mathcal{T}_1} $  by 
$$\rho_1 (x ,u) :=  \big(\rho (\mathcal{T}_0(x), u) + \rho (x , \mathcal{T}_1(u)) +\lambda  \rho(x , u)\big), \quad \text{for } x \in L_0 , u \in L_1  .$$
One can easily verify that $ \rho_1 $ defines a representation of $(L_0, {[\cdot,\cdot]}_{\mathcal{T}_0})$ on $L_1$. Moreover, for any $ x \in L_0 $ and $ u, v \in L_1 $, we have
\begin{align*}
(\mathrm{id}_{H^{\otimes 2}} \otimes_H t)(\rho_1 (x , u)) &= (\mathrm{id}_{H^{\otimes 2}} \otimes_H t)\Big( \rho (\mathcal{T}_0(x) ,u) + \rho(x , \mathcal{T}_1(u)) +\lambda \rho( x  ,  u)  \Big) \\
&=\Big( {[\mathcal{T}_0(x) * t(u)]}_{L_0} + {[x * t(\mathcal{T}_1(u))]}_{L_0} +\lambda {[x * t(u)]}_{L_0} \Big) \\&=\Big( {[\mathcal{T}_0(x) * t(u)]}_{L_0} + {[x * \mathcal{T}_0(t(u))]}_{L_0} +\lambda {[x * t(u)]}_{L_0} \Big) \\
&= {[x * t(u)]}^{\mathcal{T}_0}_{L_0},
\end{align*}
and
\begin{align*}
\rho_1 (t(u) , v) &=  \Big( \rho (\mathcal{T}_0(t(u)) , v) + \rho(t(u) , \mathcal{T}_1(v))+ \lambda \rho( t(u) , v) \Big) \\
&=  \Big( {[\mathcal{T}_1(u) * v]}_{L_1} + {[u * \mathcal{T}_1(v)]}_{L_1} + \lambda {[u * v]}_{L_1} \Big) \\
&= [u * v]^{\mathcal{T}_1}_{L_1}.
\end{align*}
This shows that the quadruple $(L_0, L_1, t, \rho_1)$ satisfies the crossed module for deformed Lie $H$-pseudoalgebras.

\begin{proposition}\label{prop:directsum_rb}
Let $({L_1}_{\mathcal{T}_1}, {L_0}_{\mathcal{T}_0}, t, \rho)$ be a crossed module of Rota--Baxter Lie $H$-pseudoalgebras. Then $(L_0 \oplus L_1, \mathcal{T}_0 \oplus \mathcal{T}_1)$ is a Rota--Baxter Lie $H$-pseudoalgebra, where $L_0 \oplus L_1$ is equipped with the pseudobracket
\begin{align}\label{eq:crossedmod_rb}
\Big[(x, u) * (y, v)\Big] := \Big({[x * y]}_{L_0}, \rho(x , v) - ((12)\otimes_H \mathrm{id}_{L_1})\rho(y , u) + {[u * v]}_{L_1}\Big),
\end{align}
for $(x,u), (y,v) \in L_0 \oplus L_1$.
\end{proposition}

\begin{proof}
Since $L_0, L_1$ are both Lie $H$-pseudoalgebras and $\rho : L_0 \boxtimes L_1 \to H^{\otimes 2} \otimes_H L_1$ is an $H^{\otimes 2}$-linear map, it follows from Proposition 5.7 in \cite{Das-2025} that $L_0 \oplus L_1$ is a Lie $H$-pseudoalgebra with the pseudobracket \eqref{eq:crossedmod_rb}. Moreover, for any $(x, u), (y, v) \in L_0 \oplus L_1$, we check the Rota--Baxter identity
\begin{align*}
&\Big[(\mathcal{T}_0 \oplus \mathcal{T}_1)(x, u) * (\mathcal{T}_0 \oplus \mathcal{T}_1)(y, v)\Big] 
\\&= \Big[(\mathcal{T}_0(x), \mathcal{T}_1(u)) * (\mathcal{T}_0(y), \mathcal{T}_1(v))\Big] \\
&= \Big({[\mathcal{T}_0(x) * \mathcal{T}_0(y)]}_{L_0}, \rho(\mathcal{T}_0(x) , \mathcal{T}_1(v)) - ((12) \otimes_H \mathrm{id}_{L_1})\rho(\mathcal{T}_0(y) , \mathcal{T}_1(u)) + {[\mathcal{T}_1(u) * \mathcal{T}_1(v)]}_{L_1}\Big) \\
&= \Big((\mathrm{id}_{H^{\otimes 2}} \otimes_H \mathcal{T}_0)\Big({[\mathcal{T}_0(x) * y]}_{L_0} + {[x * \mathcal{T}_0(y)]}_{L_0} + \lambda {[x *  y]}_{L_0} \Big), \\
&\qquad (\mathrm{id}_{H^{\otimes 2}} \otimes_H\mathcal{T}_1) \Big(\rho(\mathcal{T}_0(x) , v) + \rho(x , \mathcal{T}_1(v)) + \lambda\rho(x , v) \Big) \\
&\qquad - ((12) \otimes_H \mathcal{T}_1)\Big(\rho(\mathcal{T}_0(y) , u) + \rho(y , \mathcal{T}_1(u)) +\lambda \rho(y , u) \Big) \\
&\qquad + (\mathrm{id}_{H^{\otimes 2}} \otimes_H \mathcal{T}_1)\Big({[\mathcal{T}_1(u) * v]}_{L_1} + {[u * \mathcal{T}_1(v)]}_{L_1} +\lambda {[u * v]}_{L_1} \Big)\Big) \\
&= ((23) \otimes_H (\mathcal{T}_0 \oplus \mathcal{T}_1))\Big([(x, u) * (\mathcal{T}_0 \oplus \mathcal{T}_1)(y, v)] + [(\mathcal{T}_0 \oplus \mathcal{T}_1)(x, u) * (y, v)] + \lambda[(x, u) * (y, v)]\Big).
\end{align*}
This shows that the map $\mathcal{T}_0 \oplus \mathcal{T}_1 : L_0 \oplus L_1 \to L_0 \oplus L_1$ is a Rota--Baxter operator. This proves the result.
\end{proof} 

\begin{theorem}\label{thm:cross_rb}
There is a one-to-one correspondence between strict $2$-term Rota--Baxter $L_\infty$-$H$-pseudoalgebras and crossed modules of Rota--Baxter Lie $H$-pseudoalgebras.
\end{theorem}

\begin{proof}
Let $\mathcal{L}_\mathcal{T} = (L_1 \xrightarrow{\beta_1} L_0, \beta_2, \beta_3 = 0, \mathcal{T}_0, \mathcal{T}_1, \mathcal{T}_2 = 0)= (L_1 \xrightarrow{\beta_1} L_0, \beta_2, \mathcal{T}_0, \mathcal{T}_1)$ be a strict $2$-term Rota--Baxter  $L_\infty$-$H$-pseudoalgebra. Then from  conditions  (st3), (st6), and (st10), we find that ${L_0}_{\mathcal{T}_0}=(L_0, \beta_2, \mathcal{T}_0)$ is a Rota--Baxter Lie $H$-pseudoalgebra. Next, we define a skew-symmetric $H^{\otimes 2}$-linear bracket ${[\cdot * \cdot]}_{L_1} : L_1 \boxtimes L_1 \to H^{\otimes 2} \otimes_H L_1$ by
\begin{align*}
{[u * v]}_{L_1} := \beta_2(\beta_1(u), v), \quad \text{for } u,v \in L_1.
\end{align*}
From conditions (st2), (st5), and (st7), we see that $(L_1, {[\cdot * \cdot]}_{L_1})$ is a Lie $H$-pseudoalgebra. Moreover, condition (st9) yields $\mathcal{T}_1 : L_1 \to L_1$, a Rota--Baxter operator. Hence ${L_1}_{\mathcal{T}_1}$ is also a Rota--Baxter Lie $H$-pseudoalgebra. On the other hand, the conditions (st4) and (st8) imply that $\beta_1 : {L_1}_{\mathcal{T}_1} \to {L_0}_{\mathcal{T}_0}$ is a homomorphism of Rota--Baxter Lie $H$-pseudoalgebras. Finally, we define an $H^{\otimes 2}$-linear map $\rho : L_0 \boxtimes L_1 \to H^{\otimes 2} \otimes_H L_1$ given by
\begin{align*}
\rho(x , u) := \beta_2(x, u), \quad \text{for } x \in L_0, u \in L_1.
\end{align*}
Then, (st7) and (st9) imply that $\rho$ defines a representation of ${L_0}_{\mathcal{T}_0}$ on ${L_1}_{\mathcal{T}_1}$.  Moreover, we have
\begin{align*}
\beta_1(\rho(x , u)) = \beta_1( \beta_2(x , u))= \beta_2(x , \beta_1(u)) \quad \text{and} \quad \rho(\beta_1(u) , v) = \beta_2(\beta_1(u), v) = {[u * v]}_{L_1},
\end{align*}
for $x \in L_0, u, v \in L_1$. Hence $( {L_0}_ {\mathcal{T}_0},{L_1}_ {\mathcal{T}_1}, \beta_1, \rho)$ is a crossed module of Rota--Baxter Lie $H$-pseudoalgebras.

Conversely, let $({L_1}_ {\mathcal{T}_1}, {L_0}_ {\mathcal{T}_0}, t, \rho)$ be a crossed module of Rota--Baxter Lie $H$-pseudoalgebras. Then it is easy to verify that $(L_1 \xrightarrow{t} L_0, \beta_2, \beta_3 = 0, \mathcal{T}_0, \mathcal{T}_1, \mathcal{T}_2 = 0)$ is a strict $2$-term Rota--Baxter  $L_\infty$-$H$-pseudoalgebra, where the bracket $ \beta_2 : L_i \boxtimes L_j \to H^{\otimes 2} \otimes_H L_{i+j}$ (for $0 \leq i, j \leq 1$) is given by
\begin{align*}
\beta_2(x , y) := {[x * y]}_{L_0}, \quad \beta_2(x , u) := \rho(x , u), \quad \beta_2(u , x) := -((12) \otimes_H \mathrm{id}_{L_1})\rho(x , u) \quad \text{and} \quad \beta_2(u * v) := 0,
\end{align*}
for $x, y \in L_0, u, v \in L_1$. The above two correspondences are inverse to each other. This completes the proof.
\end{proof} 

Combining Proposition~\ref{prop:directsum_rb} and Theorem~\ref{thm:cross_rb}, we get the following result.
\begin{proposition}
Let $\mathcal{L}_\mathcal{T} = (L_1 \xrightarrow{\beta_1} L_0, \beta_2, \beta_3 = 0, \mathcal{T}_0, \mathcal{T}_1, \mathcal{T}_2 = 0)$ be a strict $2$-term Rota--Baxter  $L_\infty$-$H$-pseudoalgebra. Then ${L_0 \oplus L_1}_{\mathcal{T}_0 \oplus \mathcal{T}_1}$ is a Rota--Baxter Lie $H$-pseudoalgebra with the pseudobracket given by
\begin{align*}
\Big[(x, u) * (y, v)\Big] := \Big(\beta_2(x, y), \beta_2(x, v) - ((12)\otimes_H \mathrm{id}_{L_1})\beta_2(y, u) + \beta_2(\beta_1(u), v)\Big),
\end{align*}
for $(x, u), (y, v) \in L_0 \oplus L_1$.
\end{proposition}

\begin{example}
Let ${L}_{\mathcal{T}}$ be a Rota--Baxter Lie $H$-pseudoalgebra. Then $(L_{\mathcal{T}}, L_{\mathcal{T}}, \mathrm{id}, [\cdot*\cdot])$ is a crossed module of Rota--Baxter Lie $H$-pseudoalgebras, where $[\cdot*\cdot]$ denotes the adjoint representation. Therefore, it follows that,
$$\Big( L \xrightarrow{\mathrm{id}} L, {[ \cdot * \cdot ]}_L, \beta_3 = 0, \mathcal{T}_0 = \mathcal{T}, \mathcal{T}_1= \mathcal{T}, \mathcal{T}_2= 0 \Big)$$ is a strict $2$-term Rota--Baxter  $L_\infty$-$H$-pseudoalgebra.
\end{example}

\begin{example}
Let ${L}_{\mathcal{T}_L}$ and ${M}_{\mathcal{T}_M}$ be two Rota--Baxter Lie $H$-pseudoalgebras and $f : {L}_{\mathcal{T}_L} \to {M}_{\mathcal{T}_M}$ be a morphism of Rota--Baxter Lie $H$-pseudoalgebras. Then $(ker f, L, inc, [\cdot*\cdot])$ is a crossed module of Rota--Baxter Lie $H$-pseudoalgebras, where $inc: ker f \to L$ is the inclusion map.
\end{example} 
\section*{Declarations} 
\subsection*{ Funding} This work is funded by the Second batch of the Provincial project of the Henan Academy of Sciences (No. 241819105). This paper is also sponsored by NNSFC (No.12471038, No.12171129) and ZJNNF (No.Z25A010006).
\subsection*{Data availability }No data sets were generated or analyzed during the current study
\subsection*{Author contribution} All the authors contributed equally to this work.
\subsection*{Conflict of interest/Competing interests}The authors declare that they have no conflict of interest and
competing interest.

\end{document}